\documentclass[12pt, reqno]{amsart}

\usepackage{amsfonts,amsmath,amsthm,amssymb}
\usepackage{latexsym}

\DeclareMathOperator{\trdeg}{trdeg}
\DeclareMathOperator{\spec}{Spec}

\begin{document}

{\theoremstyle{plain}
  \newtheorem{theorem}{Theorem}[section]
  \newtheorem{corollary}[theorem]{Corollary}
  \newtheorem{proposition}[theorem]{Proposition}
  \newtheorem{lemma}[theorem]{Lemma}
  \newtheorem{question}[theorem]{Question}
  \newtheorem{conjecture}[theorem]{Conjecture}
  \newtheorem{claim}[theorem]{Claim}
}

{\theoremstyle{definition}
  \newtheorem{definition}[theorem]{Definition}
  \newtheorem{remark}[theorem]{Remark}
  \newtheorem{example}[theorem]{Example}

}

\numberwithin{equation}{section}
\def\QQ{\mathbb{Q}}
\def\ZZ{\mathbb{Z}}
\def\NN{\mathbb{N}}
\def\R{\mathcal{R}}
\def \a {\alpha}
\def \A {\mathcal{A}}
\def \b {\beta}
\def \B {\mathcal{B}}
\def \d {\delta}
\def \D {\Delta}
\def \e {\epsilon}
\def \l {\lambda}
\def \g {\gamma}
\def \G {\Gamma}
\def \H {\mathcal{H}}
\def \p {\pi}
\def \n {\nu}
\def \m {\mu}
\def \s {\sigma}
\def \S {\Sigma}
\def \w {\omega}
\def \F {\Phi}
\def \O {\mathcal{O}}
\def \W {\Omega}
\def \k {\bold{k}}
\def \bp {\bold{p}}
\def \U {\mathcal{U}}
\def \t {\tau}
\def \P {\Psi}
\def \et {\eta}
\def \th {\theta}

\def \ra {\rightarrow}

\title[Toroidalization of generating sequences]{Toroidalization of generating
sequences in dimension two function fields of positive
characteristic}

\author{Laura Ghezzi}
\address{Florida International University, Department of Mathematics,
University Park, Miami, FL 33199, USA} \email{ghezzil@fiu.edu}
\urladdr{http://www.fiu.edu/$\sim$ghezzil/}

\author{Olga Kashcheyeva}
\address{University of Illinois at Chicago, Department of Mathematics,
Statistics and Computer Science, 851 S. Morgan (m/c 249), Chicago,
IL 60607, USA} \email{olga@math.uic.edu}
\urladdr{http://www.math.uic.edu/$\sim$olga/}

\thanks{The first author is partially supported by the
Florida International University Faculty Research Award.}


\begin{abstract}
We give a characteristic free proof of the main result of \cite{GHK}
concerning toroidalization of generating sequences of valuations in
dimension two function fields. We show that when an extension of two
dimensional algebraic regular local rings $R\subset S$ satisfies the
conclusions of the Strong Monomialization theorem of Cutkosky and
Piltant, the map between generating sequences in $R$ and $S$ has a
toroidal structure.
\end{abstract}

\maketitle

\section{Introduction}

The aim of this paper is to prove the main result of \cite{GHK} in
positive characteristic. We start by recalling the set-up and the
necessary definitions.

Let $\k$ be an algebraically closed field, and let $K$ be an
algebraic function field over $\k$. We say that a subring $R$ of $K$
is {\it algebraic} if $R$ is essentially of finite type over $\k$.
We will denote the maximal ideal of a local ring $R$ by $m_R$.

Let $K^*/K$ be a finite separable extension of algebraic function
fields of transcendence degree 2 over $\k$. Let $\n^*$ be a
$\k$-valuation of $K^*$ with valuation ring $V^*$ and value group
$\G^*$. Let $\n$ be the restriction of $\n^*$ to $K$ with valuation
ring $V$ and value group $\G$. Consider an extension of algebraic
regular local rings $R \subset S$ where $R$ has quotient field $K$,
$S$ has quotient field $K^*$, $R$ is dominated by $S$ and $S$ is
dominated by $V^*$ (i.e., $m_{V} \cap R = m_R$ and $m_{V^*} \cap S =
m_S$).

Let $\F = \nu(R \backslash \{0\})$ be the semigroup of $\G$
consisting of the values of nonzero elements of $R$. For $\gamma \in
\F$, let $I_{\gamma}=\{f\in R \mid \ \n(f)\geq \gamma\}$. A
(possibly infinite) sequence $\{ Q_i \}$ of elements of $R$ is a
{\it generating sequence} of $\nu$ \cite{Spi} if for every $\gamma
\in \F$ the ideal $I_\g$ is generated by the set
$$\{\prod_{i}{Q_i}^{a_i}\mid \ a_i\in \mathbb{N}_0,\
\sum_{i} a_i \n(Q_i)\geq \gamma\}.$$ A generating sequence of $\n$
is {\it minimal} if none of its proper subsequences is a generating
sequence of $\n$.
A generating sequence of $\n^*$ in $S$ can be defined similarly.

Generating sequences provide a very useful tool in the study of
algebraic surfaces (cf. \cite{Cossart, C&P, FJ, GK, Spi, T} and the
literature cited there).

Let $(u,v)$ be a regular system of parameters (s.o.p.) of $R$, and
let $R' = R\big[ \frac{u}{v} \big]_m$, where $m$ is a prime ideal of
$R\big[ \frac{u}{v} \big]$ such that $m\cap R=m_R$. We say that $R
\to R'$ is a {\it quadratic} transform. If furthermore $\n$
dominates $R'$ we say that $R \to R'$ is a quadratic transform {\it
along $\n$}.

The main result of \cite{GHK}, which we recall below, gives a nice
structure theorem for generating sequences of $\n$ and $\n^*$, when
$\k$ has characteristic zero. We refer to Section 2 of this paper or
to Section 2 of \cite{GHK} for the precise definition of toroidal
structure.

\begin{theorem} \cite[8.1]{GHK}\label{char0}
Let $\k$ be an algebraically closed field of characteristic $0$, and
let $K^*/K$ be a finite extension of algebraic function fields of
transcendence degree $2$ over $\k$. Let $\n^*$ be a $\k$-valuation
of $K^*$ with valuation ring $V^*$, and let $\n$ be the restriction
of $\n^*$ to $K$. Suppose that $R \subset S$ is an extension of
algebraic regular local rings with quotient fields $K$ and $K^*$
respectively, such that $V^*$ dominates $S$ and $S$ dominates $R$.
Then there exist sequences of quadratic transforms $R \to \bar{R}$
and $S \to \bar{S}$ along $\n^*$ such that $\bar{S}$ dominates
$\bar{R}$ and the map between generating sequences of $\n$ and
$\n^*$ in $\bar{R}$ and $\bar{S}$ respectively, has a toroidal
structure.
\end{theorem}



The goal of this paper is to find a toroidal structure for
generating sequences of $\n$ and $\n^*$ when $\k$ has characteristic
$\bp>0$.

Cutkosky and Piltant proved that Strong Monomialization holds in
positive characteristic, provided that $V^*/V$ is defectless
\cite[7.3, 7.35]{C&P}. The {\it defect} is an invariant of
ramification theory of valuations, and it is a power of $\bp$. We
refer the reader to Section 7.1 of \cite{C&P} for the precise
definition. We have that $V^*/V$ is defectless whenever $\G^*$ (and
$\G$) are finitely generated \cite[7.3]{C&P}. The only case in which
$\G^*$ is not finitely generated is when it is a non-discrete
subgroup of $\mathbb{Q}$. Furthermore, $V^*/V$ is always defectless
when $\k$ has characteristic zero. Strong Monomialization may not
hold if the extension $V^*/V$ has a defect. See \cite[7.38]{C&P} for
an example. Since in our work we apply Strong Monomialization, we
need to assume that $V^*/V$ is defectless.

When $\G^*$ is a non-discrete subgroup of $\mathbb{Q}$ (which is the
essential and subtle case), Strong Monomialization states that there
exist sequences of quadratic transforms $R \to R_1$ and $S \to S_1$
along $\n^*$ such that $\n^*$ dominates $S_1$, $S_1$ dominates
$R_1$, and there are regular parameters $(u,v)$ in $R_1$ and $(x,y)$
in $S_1$, such that the inclusion $R_1\subset S_1$ is given by
\begin{equation}\label{introequation}
\begin{array}{ll}
u &=   x^t \d \\
v &=  y,\\
\end{array}
\end{equation}
where $t$ is a positive integer and $\d$ is a unit in $S_1$.


Observe that we can choose $u,v \in R_1$ (resp. $x,y \in S_1$) to be
the first two members of a generating sequence of $\n$ (resp.
$\n^*$). Therefore, (\ref{introequation}) exhibits a toroidal
structure of the map between the first two elements of such
generating sequences.

The definition of toroidal structures of generating sequences of
$\n$ and $\n^*$ is given in Section \ref{statement}. Our main
theorem is stated as follows.

\begin{theorem}[Theorem \ref{monomialization}] \label{intro-main}
Let $\k$ be an algebraically closed field of characteristic $\bp>0$,
and let $K^*/K$ be a finite separable extension of algebraic
function fields of transcendence degree $2$ over $\k$. Let $\n^*$ be
a $\k$-valuation of $K^*$ with valuation ring $V^*$, and let $\n$ be
the restriction of $\n^*$ to $K$, with valuation ring $V$. Assume
that $V^*/V$ is defectless. Suppose that $R \subset S$ is an
extension of algebraic regular local rings with quotient fields $K$
and $K^*$ respectively, such that $V^*$ dominates $S$ and $S$
dominates $R$. Then there exist sequences of quadratic transforms $R
\to \bar{R}$ and $S \to \bar{S}$ along $\n^*$ such that $\bar{S}$
dominates $\bar{R}$ and the map between generating sequences of $\n$
and $\n^*$ in $\bar{R}$ and $\bar{S}$ respectively, has a toroidal
structure.
\end{theorem}

We prove the theorem by analyzing the different types of valuations
of $K^*$. In most cases, the result follows from a standard
application of the Strong Monomialization theorem. These cases are
analyzed in Section \ref{easycases}. The rest of the paper is
devoted to the essential case, when $\G^*$ is a non-discrete
subgroup of $\mathbb{Q}$. We will briefly describe below the main
steps of the proof in this case.

We first remark that the methods of \cite{GHK} can not be extended
to positive characteristic. The main obstruction is that the ``key
lemma'' \cite[8.2]{GHK} no longer holds; that is, the strong
monomial form may not be preserved when we apply the quadratic
transforms of \cite[8.2]{GHK}. In this paper we use the sequences of
quadratic transforms of the ``algorithm'' described in Section 7.4
of \cite{C&P}. This algorithm is recalled in Section
\ref{strongmonomialization}. The essential point is that the strong
monomial form $(\ref{introequation})$ is eventually ``stable'' along
the algorithm (see Theorem \ref{strongmonom}).

Other technical difficulties arise from the fact that we used
\'etale extensions in several places in \cite{GHK}, but such
extensions are no longer regular rings when we work in positive
characteristic. Therefore we do not use \'etale extensions in this
paper.

Let $(u,v)$ be a regular system of parameters in $R$ and let
$\{\d_i\}_{i>0}\subset R$ be a sequence of units such that the
residue of $\d_i$ is 1 for all $i>0$. In Section \ref{jumpingpoly}
we construct a sequence of {\it jumping polynomials} $\{ T_i \}_{i
\geq 0}$ in $R$ corresponding to $(u,v)$ and to the units
$\{\d_i\}_{i>0}$. By normalizing, we may assume that $\n(u) = 1$. We
let $T_0=u$ and $T_1=v$. Write $\n(v) = p_1/q_1$, where $p_1$ and
$q_1$ are coprime positive integers. For each $i \ge 1$, we define
$T_{i+1}$ recursively. Let $p_{i+1}$ and $q_{i+1}$ be the coprime
positive integers defined by
$$\n(T_{i+1}) = q_i \n(T_i) + \dfrac{1}{q_1 \dots q_i}\cdot
\dfrac{p_{i+1}}{q_{i+1}}.$$ This construction is a generalization of
the one given in \cite{GHK}, where the sequence of units was
trivial; that is, $\d_i=1$ for all $i>0$. By allowing non trivial
units in the jumping polynomials we recover some useful results of
\cite{GHK}, avoiding the use of \'etale extensions. Jumping
polynomials corresponding to a trivial sequence of units are very
similar to Favre and Jonsson's {\it key polynomials} \cite{FJ},
whereas the idea of key polynomials is originally due to MacLane
\cite{MacL}.

We observe that the above collection of jumping polynomials $\{ T_i
\}_{i\geq 0}$ forms a generating sequence of $\n$ in $R$ (Theorem
\ref{nondiscretesequence}). Furthermore, we can select a subsequence
that forms a minimal generating sequence (Theorem
\ref{nondiscreteminsequence}).

Our proof of Theorem \ref{intro-main} proceeds as follows. We may
assume that $R$ has regular parameters $(u,v)$, and $S$ has regular
parameters $(x,y)$ such that the inclusion $R\subset S$ satisfies
\begin{equation}
\begin{array}{ll}
u &=   x^t \d \\
v &=  y,\\
\end{array}
\end{equation}
where $t$ is a positive integer and $\d$ is a unit in $S$.

We consider the sequence $\{T_i\}_{i\ge 0}$ of jumping polynomials
in $R$ corresponding to $(u,v)$ and to the trivial sequence of
units. We consider the sequence $\{T'_i\}_{i\ge 0}$ of jumping
polynomials in $S$ corresponding to $(x,y)$ and to the sequence of
units given by appropriate powers of $\d$. Then $\{T'_i\}_{i\ge 0}$
forms a generating sequence of $\n^*$ in $S$.

Let $Q_k=q_1\cdots q_k$ for $k>0$. We show in Theorem
\ref{relationship} that if $Q_k$ and $t$ are relatively prime for
all $k>0$, then $T_i=T'_i$ for all $i>0$. The theorem is proved in
this case.

Otherwise we construct appropriate sequences of quadratic transforms
$R \to R'$ and $S \to S'$, such that the inclusion $R'\subset S'$
contradicts the stable form of strong monomialization. This step is
the main part of our argument, and it requires a very explicit
description of the quadratic transforms that we perform. Several
crucial preparatory results are discussed in Section \ref{qt}.

Last we remark that the proof of Theorem \ref{intro-main} applies
also when $\k$ has characteristic zero, thus providing an
alternative argument for Theorem \ref{char0}.


\section{Statement of the result} \label{statement}

Let $\k$ be an algebraically closed field of characteristic $\bp>0$
and let $K^*/K$ be a finite separable extension of algebraic
function fields of transcendence degree 2 over $\k$. Let $\n^*$ be a
$\k$-valuation of $K^*$ with valuation ring $V^*$ and value group
$\G^*$ and let $\n$ be the restriction of $\n^*$ to $K$ with
valuation ring $V$ and value group $\G$.

Suppose that $S$ is an algebraic regular local ring with quotient
field $K^*$ which is dominated by $V^*$ and $R$ is an algebraic
regular local ring with quotient field $K$ which is dominated by
$S$. We will show that there exist sequences of quadratic transforms
$R\ra R'$ and $S\ra S'$ along $\n^*$ such that $S'$ dominates $R'$
and the map between generating sequences of $S'$ and $R'$ has the
following toroidal structure (cf. Section 2 of \cite{GHK}).

\begin{itemize}

\item[(1)] If $\n^*$ is divisorial then $R'=V$ and $S'=V^*$ with
regular parameters $u\in R'$ and $x\in S'$ such that $u=x^a\g$ for
some unit $\g\in S'$ and for some positive integer $a$. We also have
that $\{u\}$ is a minimal generating sequence of $\n$ and $\{x\}$ is
a minimal generating sequence of $\n^*$.

\item[(2)] If $\n^*$ has rank 2 then there exist regular
parameters $(u,v)$ in $R'$ and $(x,y)$ in $S'$ such that $\{u,v\}$
is a minimal generating sequence of $\n$, $\{x,y\}$ is a minimal
generating sequence of $\n^*$, and
\begin{align*}
u &=x^ay^b\d\\
v &=y^d\g
\end{align*}
for some units $\d,\g\in S'$, and for some nonnegative integers $a,
b, d$ such that $ad\neq 0$.

\item[(3)] If $\n^*$ has rank 1 and rational rank 2 then there exist regular
parameters $(u,v)$ in $R'$ and $(x,y)$ in $S'$ such that $\{u,v\}$
is a minimal generating sequence of $\n$, $\{x,y\}$ is a minimal
generating sequence of $\n^*$, and
\begin{align*}
u &=x^ay^b\d\\
v &=x^cy^d\g
\end{align*}
for some units $\d,\g\in S'$, and for some nonnegative integers $a,
b, c, d$ such that $ad-bc \neq 0$.

\item[(4)] If $\G$ and $\G^*$ are non-discrete subgroups of
$\mathbb{Q}$ and $V^*/V$ is defectless, then there exist a
generating sequence $\{H_l\}_{l\ge 0}$ of $\n$ in $R'$ and regular
parameters $(x,y)$ in $S'$ such that
\begin{align*}
H_0 &=x^a\g\\
H_1 &=y
\end{align*}
for some unit $\g\in S'$ and for some positive integer $a$, and
$\{x,\{H_l\}_{l>0}\}$ is a minimal generating sequence of $\n^*$ in
$S'$.

\item[(5)] If $\n$ is discrete but not divisorial then there exist
regular parameters $(u,v)$ in $R'$ and $(x,y)$ in $S'$ such that
$\G$ is generated by $\n(u)$, $\G^*$ is generated by $\n^*(x)$, and
$u=x^a\g$ for some unit $\g\in S'$ and for some positive integer
$a$. Moreover, $R'$ has a non-minimal generating sequence
$\{u,\{T_i\}_{i>0}\}$  such that $\{x,\{T_i\}_{i>0}\}$ is a
non-minimal generating sequence in $S'$.

\end{itemize}


\section{Valuations in 2-dimensional function fields} \label{easycases}
We will prove our main theorem by analyzing the different types of
valuations of $K^*$. A similar analysis was done in \cite{GHK}
(Section 3), but we outline it below for completeness. We refer to
\cite{C&P} (Section 7.2) for the background needed in this section.
We recall that $\G$ and $\G^*$ are finitely generated except when
they are isomorphic to non-discrete subgroups of $\mathbb{Q}$.

\subsection{One dimensional valuations}
By definition, $\n^*$ is divisorial. In this case $\n$ and $\n^*$
are discrete, and $V$ and $V^*$ are iterated quadratic transforms of
$R$ and $S$ respectively (see \cite[4.4]{Ab}).

Let $u$ be a regular parameter of $V$ and let $x$ be a regular
parameter of $V^*$. Then there is a relation $$u=x^a\g$$ where
$\g\in V^*$ is a unit and $a$ is a positive integer. Since $\{ u \}$
is a mimimal generating sequence for $V$, and $\{ x \}$ is a minimal
generating sequence for $V^*$ the theorem is proved.

\subsection{Zero dimensional valuations of rational rank 2}

By \cite[7.3]{C&P} there exist sequences of quadratic transforms
$R\ra R'$ and $S\ra S'$ along $\n^*$ such that $R'$ has regular
parameters $(u,v)$, $S'$ has regular parameters $(x,y)$, and
\begin{align*}
u &=x^ay^b\d\\
v &=x^cy^d\g
\end{align*}
for some units $\d,\g\in S'$ and for some nonnegative integers $a,
b, c, d$ such that $ad-bc \neq 0$. Further, $c=0$ if $\n^*$ has rank
two. We also have that $\{\n(u), \n(v)\}$ is a rational basis of
$\G\otimes \mathbb{Q}$, and $\{\n^*(x), \n^*(y)\}$ is a rational
basis of $\G^*\otimes \mathbb{Q}$.

Fix $\g\in \F =\nu(R' \backslash \{0\})$ and let
$I_{\gamma}=\{f\in R' \mid \ \n(f)\geq \gamma\}$. If $f\in
I_{\gamma}$ we can write $f=\sum_{i\ge 1}a_iu^{b_i}v^{c_i}$, where
$a_i$ are units in $R'$, $b_i$ and $c_i$ are nonnegative integers,
and the terms have increasing value, since $\n(u)$ and $\n(v)$ are
rationally independent. It follows that $\n(f)=b_1\n(u)+c_1\n(v)$.
For $i\geq 1$ we have $b_i\n(u)+c_i\n(v)\geq
b_1\n(u)+c_1\n(v)=\n(f)\geq \g$. Therefore $f$ belongs to the ideal
generated by the set $\{{u}^{b_i}{v}^{c_i}\mid \ b_i, c_i\in
\mathbb{N}_0,\ b_i\n(u)+c_i\n(v)\geq \gamma\}$. This implies that
$\{ u,v \}$ is a generating sequence of $\n$ in $R'$. Furthermore,
it is minimal. Similarly $\{ x,y \}$ is a minimal generating
sequence of $\n^*$ in $S'$, and the theorem is proved.


\medskip

The rest of the paper will be devoted to studying the remaining
cases, that is zero dimensional valuations of rational rank 1.

\subsection{Non-discrete zero dimensional valuations of rational rank 1}

We can normalize $\G^*$ so that it is an ordered subgroup of
$\mathbb{Q}$, whose denominators are not bounded, as $\G^*$ is not
discrete. In Example 3, Section 15, Chapter VI of \cite{ZS},
examples are given of two-dimensional algebraic function fields with
value group equal to any given subgroup of the rational numbers.
This case is much more subtle.

\subsection{Discrete zero dimensional valuations of rational rank 1}

If $\n^*$ is discrete, then $\n$ is also discrete. This case will be
handled in the same way as the case of non-discrete zero dimensional
valuations of rational rank 1, but the generating sequences of
$\n^*$ and $\n$ will not be minimal.


\section{Jumping polynomials}\label{jumpingpoly}
Throughout this section we work under the assumption that the value
group of $\n$ is a subgroup of $\mathbb{Q}$ and
$\trdeg_{\k}(V/{m_{V}})=0$. We will first generalize the
construction of a sequence of jumping polynomials given in
\cite{GHK}.

Suppose that $(u,v)$ is a system of regular parameters in $R$ and
$\{\d_i\}_{i>0}\subset R$ is a sequence of units such that  the
residue of $\d_i$ is 1 for all $i>0$. Suppose also that the value
group $\G$ is normalized so that $\n(u)=1$. Let
\begin{align*}
\left\{ \begin{array}{ccc}
T_0 & = & u\\
T_1 & = & v.
\end{array} \right.
\end{align*}
Set $q_0=\infty$ and choose a pair of coprime positive integers
$(p_1,q_1)$ so that $\n(v)=p_1/q_1$. For $i \ge 1$, $T_{i+1}$ is
defined recursively as follows. Let
$$T_{i+1}=T_i^{q_i}-\l_i\d_i\prod_{j=0}^{i-1}T_j^{n_{i,j}},$$
where $n_{i,j}<q_j$ are nonnegative integers such that
$q_i\nu(T_i)=\nu(\prod_{j=0}^{i-1}T_j^{n_{i,j}})$ and $\l_i\in
\k-\{0\}$ is the residue of
${T_i^{q_i}}{(\prod_{j=0}^{i-1}T_j^{n_{i,j}})^{-1}}$.

Writing $\d_i=1+w_i$, for some $w_i\in m_R$, we notice that
$$
\n(T_{i+1})\geq \min
\{\n(T_i^{q_i}-\l_i\prod_{j=0}^{i-1}T_j^{n_{i,j}}), \n(\l_i
w_i\prod_{j=0}^{i-1}T_j^{n_{i,j}})\}>q_i\n(T_i).
$$
Therefore we can choose positive integers $p_{i+1}$ and $q_{i+1}$ so
that $(p_{i+1},q_{i+1})=1$ and
$$\n(T_{i+1})=q_i\n(T_i)+\dfrac{1}{q_1\cdots
q_i}\cdot\dfrac{p_{i+1}}{q_{i+1}}.$$

We will say that $\{T_i\}_{i\ge 0}$ is a sequence of jumping
polynomials corresponding to the regular parameters $(u,v)$ and the
sequence of units $\{\d_i\}_{i>0}$. The polynomial $T_{i}$ will be
called the $i$-th {\it jumping polynomial} and the value $\n(T_i)$
will be called the $i$-th j-{\it value}. We denote the $i$-th
j-value by $\b_i$ and we say that $\b_i$ is an {\it independent}
j-value if $q_i\neq 1$. In this case we say that $T_i$ is an
independent jumping polynomial.

It is shown in \cite[5.10]{GHK} that the sequence of jumping
polynomials corresponding to the trivial sequence of units
$\{1\}_{i>0}$ is well defined, that is, the $n_{i,j}$ above are
uniquely determined. The same considerations show that the sequence
of jumping polynomials $\{T_i\}_{i\ge 0}$ corresponding to any
sequence of units $\{\d_i\}_{i>0}\subset R$, where $(\d_i-1)\in m_R$
for all $i>0$, is well defined. Furthermore, the values of jumping
polynomials have the following properties.

\begin{remark}\label{b-inequality}
For $i>0$ denote $Q_i=q_1\cdots q_i$ and set $Q_0=1$. If $i>0$, then
\begin{itemize}
\item[1)]$\b_{i+1}=q_i\b_i+\dfrac{1}{Q_i}\cdot\dfrac{p_{i+1}}{q_{i+1}}$,

\item[2)] $Q_i\b_j$ is an integer number for all $j\le i$,
\smallskip

\item[3)]$q_{i+1}\b_{i+1}\geq\b_{i+1}>q_i\b_i\ge\b_i$,
\smallskip

\item[4)]$q_i\b_i=\sum_{j=0}^{i-1}n_{i,j}\b_j$.
\end{itemize}
\end{remark}

Assume now that $\{\b_{i_l}\}_{l\geq 0}$ is the subsequence of all
independent j-values. Let $\bar{\b}_l=\b_{i_l}$ denote the $l$-th
independent j-value, $\bar{q}_l=q_{i_l}$ and
$\bar{p}_l=(p_{i_{l-1}+1}+\dots+p_{i_l-1})\bar{q}_l+p_{i_l}$ if
$l>0$. Then the values of independent jumping polynomials have the
following properties.

\begin{remark}\label{barb-inequality}
For $l>0$ denote $\bar{Q}_l=\bar{q}_1\cdots \bar{q}_l$ and set
$\bar{Q}_0=1$. If $l>0$, then
\begin{itemize}

\item[1)]$\bar{\b}_{l+1}=\bar{q}_l\bar{\b}_l+\dfrac{1}{\bar{Q}_l}
\cdot\dfrac{\bar{p}_{l+1}}{\bar{q}_{l+1}}$ and
$\bar{\b}_1=\dfrac{\bar{p}_1}{\bar{q}_1}$,

\item[2)] $\bar{Q}_l\b_{i'}$ is an integer number for all $i'<
i_{l+1}$. In particular, $\bar{Q}_l\bar{\b}_j$ is an integer number
for all $j\le l$,
\smallskip

\item[3)]$\bar{q}_{l+1}\bar{\b}_{l+1}>\bar{\b}_{l+1}>\bar{q}_l\bar{\b}_l>\bar{\b}_l$,
\smallskip

\item[4)]$(\bar{p}_l,\bar{q}_l)=1$.
\end{itemize}
\end{remark}

Let us consider the sequence $\{H_l\}_{l\ge0}$ of all independent
jumping polynomials in $R$. For all $l\ge 0$ we have $H_l=T_{i_l}$,
in particular, $H_0=u$ and
$H_1=v-\sum_{j=1}^{i_1-1}\l_j\d_ju^{\b_j}$. Since $n_{i,j}<q_j$,
that is, $n_{i,j}=0$ whenever $T_j$ is not an independent jumping
polynomial, the recursive formula for $H_{l+1}$ with $l>0$ is
$$H_{l+1}
=H_l^{\bar{q}_l}-\l_{i_l}\d_{i_l}\prod_{j=0}^{l-1}H_j^{n_{i_l,i_j}}-
\l_{i_l+1}\d_{i_l+1}\prod_{j=0}^l
H_j^{n_{i_l+1,i_j}}-\l_{i_l+2}\,\d_{i_l+2}\prod_{j=0}^l
H_j^{n_{i_l+2,i_j}}- \dots$$ $$\dots
-\l_{i_{l+1}-1}\,\d_{i_{l+1}-1}\prod_{j=0}^l H_j^{n_{i_{l+1}-1,i_j}}
=H_l^{\bar{q}_l}-\l_{i_l}\d_{i_l}\prod_{j=0}^{l-1}H_j^{n_{i_l,i_j}}-
\sum_{i'=i_l+1}^{i_{l+1}-1}\l_{i'}\d_{i'}\prod_{j=0}^l
H_j^{n_{i',i_j}}.$$

\begin{remark}\label{normalize} In general, if $R$ is a
2-dimensional regular local ring dominated by $V$ and  $(u,v)$ is a
system of regular parameters in $R$, we may not necessarily have
$\n(u)=1$. Then in order to define a sequence of jumping polynomials
$\{T_i\}_{i\ge 0}$ corresponding to the system of regular parameters
$(u,v)$, we introduce the following valuation $\tilde\n$ of $K$
$$
\tilde\n(f)=\frac{\n(f)}{\n(u)}
$$
for all $f\in K$. Then $\tilde\n(u)=1$ and we use the construction
above with $\n$ replaced by the equivalent valuation $\tilde\n$.
This procedure is equivalent to normalizing the value group $\G$ so
that $\n(u)=1$.
\end{remark}

\section{Properties of jumping polynomials}
In this section assumptions and notations are as in Section
\ref{jumpingpoly}. Our first goal is to show that sequences of
jumping polynomials form generating sequences of valuations. See
\cite{C&P,FJ,MacL,Spi} for more considerations on this topic.

Let $\F = \n(R\backslash\{0\})$ be the semigroup of $\G$
consisting of the values of nonzero elements of $R$. For $\g\in
\F$, let $I_\g=\{f\in R|\ \n(f)\ge\g\}$. Then a possibly infinite
sequence $\{Q_i\}\subset R$ is a {\it generating sequence} of $\n$
if for every $\g\in\F$ the ideal $I_\g$ is generated by the set
$\{\prod_i{Q_i}^{a_i}| \ a_i\in\NN_0,\ \sum_{i} a_i \n(Q_i)\ge\g\}$.
A generating sequence of $\n$ is {\it minimal} if none of its proper
subsequences is a generating sequence of $\n$.

If $\g\in\F$ we denote by $\A_\g$ the ideal of $R$ generated by
$\{\prod_{j=0}^k T_j^{m_j}|\ k,m_j\in\NN_0,\\
 \sum_{j=0}^k m_j\b_j\ge\g\}$, and we denote by $\A^+_\g$ the ideal of $R$
generated by $\{\prod_{j=0}^k T_j^{m_j}|\ k,m_j\in\NN_0,\
\sum_{j=0}^k m_j\b_j>\g\}$. We observe the following basic
properties of the ideals $\A_\g$ and $\A^+_\g$.
\begin{itemize}

\item[1)] $\A_\g\subset I_\g$.

\item[2)] If $\g_1<\g_2$ then
$\A_{\g_2}\subset\A^+_{\g_1}\subset\A_{\g_1}$.

\item[3)] $\A_{\g_1}\A_{\g_2}\subset\A_{\g_1+\g_2}$ and
$\A^+_{\g_1}\A_{\g_2}\subset\A^+_{\g_1+\g_2}$.

\item[4)] Let $\b=\min(\b_0,\b_1)$. If $f\in m_R$ then
$f=uf_1+vf_2$ for some $f_1,f_2\in R$ and therefore $f\in\A_\b$.
Thus $m_R\subset\A_\b\subset\A^+_0$.

\item[5)] For any $\g\in\F$ there exists $\g'\in\F$ such that
$\A^+_\g=\A_{\g'}$. Indeed, notice that the set
$\{\a\in\F|\g<\a\le\g+\b_0\}$ is finite, since it is bounded from
above, and it is nonempty. Then set
$\g'=\min\{\a\in\F|\g<\a\le\g+\b_0\}$.

\end{itemize}

\begin{lemma}\label{reduced}
Let $f=\prod_{j=0}^{k}T_j^{m_j}$, where $k, m_0,\dots,m_k \in
\NN_0$, and let $\g=\n(f)=\sum_{j=0}^{k}m_j\b_j$. There exist
nonnegative integers $d_0, d_1,\dots, d_k$ such that $\sum_{j=0}^k
d_j\b_j=\g$ and $d_j<q_j$ for all $0\le j\le k$, a unit $\m\in R$
and $f'\in \A_{\g}^+$ such that $f=\m\prod_{j=0}^{k}T_j^{d_j}+f'$.
\end{lemma}

\begin{proof}
We apply induction on $k$. If  $k=0$ then $f=T_0^{m_0}$ is the
required presentation. If $k>0$, write $m_k=rq_k+d_k$ for some $r\ge
0$ and $0\le d_k<q_k$. Recall that
$\n(\prod_{j=0}^{k-1}T_j^{n_{k,j}})=q_k\b_k$ and
$\n(T_{k+1})>q_k\b_k$. Thus
$$
T_k^{rq_k}=(T_{k+1}+\l_k\d_k\prod_{j=0}^{k-1}T_j^{n_{k,j}})^r=
\l_k^r\d_k^r\prod_{j=0}^{k-1}T_j^{rn_{k,j}}+h,
$$
where $h\in\A^+_{rq_k\b_k}$. Furthermore, since
$d_k\b_k+\sum_{j=0}^{k-1}m_j\b_j=\g-rq_k\b_k$ we have that
$hT_k^{d_k}\prod_{j=0}^{k-1}T_j^{m_j}\in\A^+_{rq_k\b_k}\A_{\g-rq_k\b_k}
\subset\A^+_\g$.

Let $g=\prod_{j=0}^{k-1}T_j^{m_j+rn_{k,j}}$ and let $\a=\g-d_k\b_k$.
Notice that $\n(g)=\a$. Then by the inductive assumption there exist
nonnegative integers $d_0, d_1,\dots, d_{k-1}$ such that
$\sum_{j=0}^{k-1}d_j\b_j=\a$ and $d_j<q_j$ for all $0\le j\le k-1$,
a unit $\m'\in R$ and $g'\in\A^+_\a$ such that $
g=\m'\prod_{j=0}^{k-1}T_j^{d_j}+g'$. We also notice that
$g'T_k^{d_k}\in\A^+_\a\A_{d_k\b_k}\subset\A^+_\g$. Thus
\begin{align*}
f&=T_k^{d_k}(\l_k^r\d_k^r\prod_{j=0}^{k-1}T_j^{rn_{k,j}}+h)\prod_{j=0}^{k-1}T_j^{m_j}=
T_k^{d_k}(\l_k^r\d_k^rg+h\prod_{j=0}^{k-1}T_j^{m_j})=\\
&=\l_k^r\d_k^r\m'\prod_{j=0}^{k}T_j^{d_j}+\l_k^r\d_k^rg'T_k^{d_k}+
hT_k^{d_k}\prod_{j=0}^{k-1}T_j^{m_j}=\m\prod_{j=0}^{k}T_j^{d_j}+f',
\end{align*}
where $\m=\l_k^r\d_k^r\m'$ is a unit in $R$ and $f'\in\A^+_\g$.
\end{proof}

\begin{remark}
The integers $d_0, d_1,\dots, d_k$ of Lemma \ref{reduced} depend
only on $\g$: there exists a unique $(k+1)$-tuple of nonnegative
integers $d_0, d_1,\dots, d_k$ such that $\sum_{j=0}^k d_j\b_j=\g$
and $d_j<q_j$ for all $0\le j\le k$.

The statement above is equivalent to the one claiming that if
$\sum_{j=0}^k c_j\b_j=0$ for some integer coefficients
$-q_j<c_j<q_j$ then $c_j=0$ for all $0\le j\le k$. We refer the
reader to Proposition 5.8 of \cite{GHK} for the proof.
\end{remark}

\begin{lemma}\label{i=a}
If $\g\in\F$, then $I_\g=\A_\g$.
\end{lemma}

\begin{proof}
We only need to check that $I_\g\subset \A_\g$ for all $\g\in\F$.

Let $\g\in\F$ and let $f\in I_\g$. We will show that $f\in\A_\g$.
First notice that if $f\in\A_\a$ for some $\a\in\F$ then
$\a\le\n(f)$. Thus the set $\W=\{\a\in\F|\,f\in\A_\a\}$ is finite
since it is bounded from above and it is nonempty since $f\in\A_0$.
We choose $\s$ to be the maximal element of $\W$. Then there exists
a presentation
$$
f=\sum_{l=1}^N g_l\prod_{j=0}^k T_j^{m_{l,j}}+f',
$$
where $\sum_{j=0}^k m_{l,j}\b_j=\s$ for all $1\le l\le N$, $g_l\in
R$ for all $1\le l\le N$, and $f'\in \A^+_\s$.

We now apply Lemma \ref{reduced} to $\prod_{j=0}^k T_j^{m_{l,j}}$
for all $1\le l\le N$. We get $ \prod_{j=0}^k
T_j^{m_{l,j}}=\m_l\prod_{j=0}^k T_j^{d_j}+h_l$, where $\m_l\in R$ is
a unit, $\sum_{j=0}^k d_j\b_j=\s$ and $h_l\in\A^+_\s$. Thus
$$
f=(\sum_{l=1}^N\m_l g_l)\prod_{j=0}^k T_j^{d_j}+\sum_{l=1}^N h_l
g_l+f'= \m\prod_{j=0}^k T_j^{d_j}+h,
$$
where $h\in\A^+_\s$. If $\m\in m_R$ then $\m\prod_{j=0}^k
T_j^{d_j}\in\A^+_0\A_\s\subset\A^+_\s$ and therefore $f\in\A^+_\s$.
Let $\a\in\F$ be such that $\A^+_\s=\A_\a$. Then $\a>\s$ and
$f\in\A_\a$, a contradiction to the choice of $\s$. So $\m$ is a
unit in $R$ and $\n(\m)=0$. Then $\n(f)=\min(\n(\m\prod_{j=0}^k
T_j^{d_j}),\n(h))=\s$. Thus we get that $\s\ge\g$, and so
$f\in\A_\s\subset\A_\g$.
\end{proof}

\begin{theorem}\label{nondiscretesequence}
$\{T_i\}_{i\ge 0}$ is a generating sequence in $R$.
\end{theorem}

\begin{proof}
The statement follows at once from Lemma \ref{i=a} and the
definition of generating sequences.
\end{proof}

\subsection{Non-discrete case.} We will now assume that the value
group of $\n$ is not finitely generated and, therefore, the sequence
of independent jumping polynomials $\{H_l\}_{l\ge 0}$ is infinite.
If $\g\in\F$ denote by $\B_\g$ the ideal of $R$ generated by
$\{\prod_{j=0}^k H_j^{m_j}|\ k,m_j\in\NN_0,\ \sum_{j=0}^k
m_j\bar{\b}_j\ge\g\}$ and denote by $\B'_\g$ the ideal of $R$
generated by $\{\prod_{j=1}^k H_j^{m_j}| \ k,m_j\in\NN_0,\
\sum_{j=1}^k m_j\bar{\b}_j\ge\g\}$. We notice that
$\B_{\g_1}\B_{\g_2}\subset\B_{\g_1+\g_2}$,
$\B'_{\g_1}\B'_{\g_2}\subset\B'_{\g_1+\g_2}$, and
$\B'_\g\subset\B_\g\subset\A_\g$ for all $\g,\g_1,\g_2\in\F$.

\begin{lemma}\label{T-to-H}
Suppose that $\G$ is a non-discrete subgroup of $\QQ$. Then
$T_k\in\B_{\b_k}$ for all $k\ge 0$. Furthermore, if $\bar{p}_1=1$
then $H_0\in\B'_{\bar{\b}_0}$.
\end{lemma}

\begin{proof}
We fix $k\ge 0$. Since $\{H_l\}_{l\ge 0}$ is infinite there exists
$l$ such that $i_{l-1}<k\leq i_l$. Then since $q_j=1$ for all $k\le
j<{i_l}$ we have
$$
T_{i_l}=T_k-\l_k\d_k\prod_{j=0}^{k-1}T_j^{n_{k,j}}-\l_{k+1}\d_{k+1}\prod_{j=0}^k
T_j^{n_{k+1,j}}-\dots-
\l_{i_l-1}\d_{i_l-1}\prod_{j=0}^{i_l-2}T_j^{n_{i_l-1,j}}.
$$
We notice that $n_{i,j}=0$ whenever $T_j$ is not an independent
jumping polynomial. Thus
$$
T_k=T_{i_l}+\sum_{i'=k}^{i_l-1}\l_{i'}\d_{i'}\prod_{j=0}^{i'-1}T_j^{n_{i',j}}
=H_l+\sum_{i'=k}^{i_l-1}\l_{i'}\d_{i'}\prod_{j=0}^{l-1}H_j^{n_{i',i_j}},
$$
where $\bar{\b}_l\ge\b_k$ and
$\sum_{j=0}^{l-1}n_{i',i_j}\bar{\b}_j=\n(\prod_{j=0}^{l-1}H_j^{n_{i',i_j}})
=\n(\prod_{j=0}^{i'-1}T_j^{n_{i',j}})=q_{i'}\b_{i'}=\b_{i'}\ge \b_k$
for all $k\le i'<i_l$. So $T_k\in\B_{\b_k}$.

Assume now that $\bar{p}_1=1$. Then since
$\bar{p}_1=(p_1+\dots+p_{i_1-1})\bar{q}_1+p_{i_1}$ we have $i_1=1$,
$1=\bar{p}_1=p_1$, $\bar{q}_1=q_1$ and $H_0=u$, $H_1=v$. Also
$$
H_2=v^{q_1}-\l_1\d_1u-\sum_{i'=2}^{i_2-1}\l_{i'}\d_{i'}u^{n_{i',0}}v^{n_{i',1}},
$$
where $n_{i',1}<q_1$ for all $2\le i'\le i_2-1$. Since
$\n(u^{n_{i',0}}v^{n_{i',1}})=\b_{i'}>q_1\b_1=1$ and
$\n(v^{n_{i',1}})=n_{i',1}/q_1<1$ we see that $n_{i',0}>0$ for all
$2\le i'\le i_2-1$. Thus
$$
H_0=u=(v^{q_1}-H_2)(\l_1\d_1+
\sum_{i'=2}^{i_2-1}\l_{i'}\d_{i'}u^{n_{i',0}-1}v^{n_{i',1}})^{-1}=(H_1^{q_1}-H_2)\D,
$$
where $\D$ is a unit in $R$ and $q_1\b_1=1$, $\b_2>1$. So
$H_0\in\B'_1=\B'_{\bar{\b}_0}$.
\end{proof}

\begin{lemma}\label{a=b}
Suppose that $\G$ is a non-discrete subgroup of $\QQ$. If
$\g\in\F$, then $\B_\g=\A_\g$. Furthermore, if $\bar{p}_1=1$ then
$\B'_\g=\A_\g$.
\end{lemma}

\begin{proof}
To prove that $\B_\g=\A_\g$ it suffices to show that if
$\sum_{j=0}^k m_j\b_j\geq\g$ then $\prod_{j=0}^k T_j^{m_j}\in\B_\g$.
By Lemma \ref{T-to-H} we have that $\prod_{j=0}^k
T_j^{m_j}\in\prod_{j=0}^k\B_{\b_j}^{m_j}\subset\B_\g$.

Now assume that $\bar{p}_1=1$ and consider $\prod_{j=0}^k H_j^{m_j}$
with $\sum_{j=0}^k m_j{\bar{\b}_j}\geq\g$. By Lemma \ref{T-to-H} we
have that $\prod_{j=0}^k
H_j^{m_j}\in\prod_{j=0}^k(\B'_{\bar{\b}_j})^{m_j}\subset\B'_\g$.
Thus $\B'_\g=\B_\g=\A_\g$.
\end{proof}

\begin{theorem}\label{nondiscreteminsequence}
If $\G$ is a non-discrete subgroup of $\QQ$ then
$\{H_l\}_{l\ge 0}$ forms a generating sequence in $R$. Moreover, if
$\bar{p}_1\neq 1$ then $\{H_l\}_{l\ge 0}$ is a minimal generating
sequence in $R$ and if $\bar{p}_1=1$ then $\{H_l\}_{l>0}$ forms a
minimal generating sequence in $R$.
\end{theorem}

\begin{proof}
It follows from Lemmas \ref{i=a} and \ref{a=b} that $I_\g=\B_\g$ for
all $\g\in\F$. Thus $\{H_l\}_{l\ge 0}$ is a generating sequence
for $\n$. If $\bar{p}_1=1$ then we have $I_\g=\B'_\g$ for all
$\g\in\F$. Thus in this case $\{H_l\}_{l>0}$ forms a generating
sequence for $\n$.

To prove the statement about minimality we introduce the following
notation: for $k\ge 0$ denote by $\bar{\G}_k$ the group generated by
$\{\bar{\b}_j\}_{j=0}^k$, denote by $\F_k$ the semigroup generated
by $\{\bar{\b}_j\}_{j=0}^k$, and denote by $\F_{\hat{k}}$ the
semigroup generated by
$\{\bar{\b}_j\}_{j=0}^{k-1}\cup\{\bar{\b}_j\}_{j>k}$. We will prove
first that if $k>0$, then $\bar{\b}_k\notin\F_{\hat{k}}$. Therefore,
$\F\neq \F_{\hat{k}}$ for all $k>0$.

It is shown in \cite[5.6]{GHK}  that $\bar{\G}_k=(1/\bar{Q}_k)\ZZ\ $
for all $k\ge 0$. Thus $\bar{\G}_{k-1}\neq\bar{\G}_k$ and we have
$\bar{\b}_k\notin\bar{\G}_{k-1}$. So $\bar{\b}_k\notin\F_{k-1}$. On
the other hand $\bar{\b}_{k+j}>\bar{\b}_k$ for all $j,k>0$, so if
$\bar{\b}_k\in\F_{\hat{k}}$ for some $k>0$ then
$\bar{\b}_k\in\F_{k-1}$. Thus $\bar{\b}_k\notin\F_{\hat{k}}$ for all
$k>0$.

It follows that if some subsequence $\H=\{H_{l_j}\}_{j\ge 0}$ is a
generating sequence for $\n$ then $\{H_l\}_{l>0}\subset\H$, since
$\{\n(H_{l_j})\}_{j\ge 0}$ needs to generate $\F$.

Assume now that $\bar{p}_1\neq 1$. Then $\bar{\b}_0$ is not a
multiple of $\bar{\b}_1$ and for all $j\ge 2$ we have $\bar{\b}_j\ge
\bar{\b}_2>\bar{q}_1\bar{\b_1}=\bar{p}_1>1=\bar{\b}_0$. Thus
$\bar{\b}_0\notin\F_{\hat{0}}$ and, therefore, any generating
sequence $\{H_{l_j}\}_{j\ge 0}$ has to contain $H_0$.
\end{proof}

\subsection{Discrete case.}
Suppose now that the value group of $\n$ is isomorphic to $\ZZ$.
Then by \cite{Spi} (p. 154) every generating sequence of $\n$ is
infinite and there are no minimal generating sequences in $R$. In
our construction we will be mostly concerned with the situation when
$R$ has a system of regular parameters $(u,v)$ such that $\n(u)$
generates $\G$. In this case sequences of jumping polynomials in $R$
have the following property.

\begin{theorem}\label{discrete}
Suppose that $\G\cong\ZZ$ and that $\b_0$ generates $\G$.
Then any infinite subsequence $\{T_{i_j}\}_{j\ge 0}$ containing
$T_0$ is a generating sequence in $R$.
\end{theorem}

\begin{proof} We will follow the same line of arguments as in Lemmas
\ref{T-to-H} and \ref{a=b}. If $\g\in\F$ denote by $\A'_\g$ the
ideal of $R$ generated by $\{\prod_{j=0}^k T_{i_j}^{m_j}|\
k,m_j\in\NN_0, \sum_{j=0}^k m_j\b_{i_j}\ge\g\}$. We will show that
$\A'_\g=\A_\g$ for all $\g\in\F$. Then Lemma \ref{i=a} implies
that $\{T_{i_j}\}_{j\ge 0}$ is a generating sequence of $\n$.

We notice first that since $\G$ is generated by $\b_0=1$, for all
$i>0$ we have $\b_i\in\NN$ and $q_i=1$. It follows that
$n_{i,0}=\b_i$ and $n_{i,i'}=0$ for all $0<i'<i$.

Fix $k\ge 0$. Since $\{T_{i_j}\}_{j\ge 0}$ is infinite there exists
$j$ such that $i_{j-1}<k\leq i_j$. Then we have
$$
T_k=T_{i_j}+\l_{i_j-1}\d_{i_j-1}T_0^{\b_{i_j-1}}+\dots+\l_k\d_kT_0^{\b_k},
$$
where $\n(T_{i_j})=\b_{i_j}\ge\b_k$ and
$\n(T_0^{\b_{i'}})=\b_{i'}\ge\b_k$ for all $k\le i'\le i_j-1$. Thus
$T_k\in\A'_{\b_k}$. To prove that $\A'_\g=\A_\g$ it suffices to show
that if $\sum_{j=0}^k m_j\b_{j}\geq\g$ then $\prod_{j=0}^k
T_{j}^{m_j}\in\A'_\g$. This is true since $\prod_{j=0}^k
T_{j}^{m_j}\in\prod_{j=0}^k(\A'_{\b_{j}})^{m_j}\subset\A'_\g$.

\end{proof}

\medskip

Our next goal is to understand the relationship between jumping
polynomials in $R$ and $S$ when the inclusion $R\subset S$ satisfies
the conclusions of the Strong Monomialization theorem. We fix
regular parameters $(u,v)$ in $R$. Let $\{T_i\}_{i\ge 0}$ be the
sequence of jumping polynomials in $R$ corresponding to $(u,v)$ and
the trivial sequence of units. For all $i>0$ let $(p_i,q_i)$ be the
pair of coprime integers defined in the construction of the jumping
polynomials, let $\l_i$ be the scalar and for $0\le j<i$ let
$n_{i,j}$ be the powers defined in the construction of
$\{T_i\}_{i\ge 0}$. We also fix regular parameters $(x,y)$ in $S$
and a unit $\d\in S$ such that $(\d-1)\in m_S$. Let $\{T'_i\}_{i\ge
0}$ be the sequence of jumping polynomials in $S$ corresponding to
$(x,y)$ and the sequence of units $\{\d^{n_{i,0}}\}_{i>0}$. For all
$i>0$ let $(p'_i,q'_i)$ be the pair of coprime integers defined in
the construction of the jumping polynomials, let $\l'_i$ be the
scalar and for $0\le j<i$ let $n'_{i,j}$ be the powers defined in
the construction of $\{T'_i\}_{i\ge 0}$.

\begin{theorem}\label{relationship}
With notations as above, suppose that the inclusion of
$2$-dimensional regular local rings $R\subset S$ satisfies the
equation
\begin{align*}
u & =x^t\d\\
v & =y,
\end{align*}
where $t$ is a positive integer. If $(t,Q_k)=1$ for some $k>0$ then
\begin{itemize}

\item[1)] $T'_i=T_i$ for all $0<i\le k+1$;

\item[2)] $q'_i=q_i$ and $p'_i=tp_i$ for all $0<i\le k$;

\item[3)] $n'_{i,j}=n_{i,j}$ for all $0<j<i\le k$ and
$n'_{i,0}=tn_{i,0}$ for all $0<i\le k$.

\end{itemize}
\end{theorem}

\begin{proof}
We may assume that $\n(u)=\n^*(u)=1$. Then $\n^*(x)=1/t$ and in
order to construct a sequence of jumping polynomials corresponding
to the regular parameters $(x,y)$ in $S$ we define the following
valuation $\tilde\n$ of $K^*$:
$$
\tilde\n(f)=t\n^*(f) \text { for all } f\in K^*.
$$
We have $T_0=u$, $T'_0=x$, $T_1=v$ and $T'_1=y$. Thus $T_0=\d
{T'_0}^t$ and $T_1=T'_1$.

Suppose that $(t,q_1)=1$. The coprime integers $p'_1$ and $q'_1$ are
such that ${p'_1}/{q'_1}=\tilde\n(y)=t\n^*(v)=tp_1/q_1$. Since
$(tp_1,q_1)=1$ we get $p'_1=tp_1$ and $q'_1=q_1$. Since
$n_{1,0}=p_1$ and $n'_{1,0}=p'_1$ we also have $n'_{1,0}=tn_{1,0}$.
Furthermore, ${T'_1}^{q'_1}=T_1^{q_1}$ and
$\d^{n_{1,0}}{T'_0}^{n'_{1,0}}=T_0^{n_{1,0}}$. Then, since the
residue of $\d$ is 1 and
${T'_1}^{q'_1}{T'_0}^{-n'_{1,0}}=T_1^{q_1}T_0^{-n_{1,0}}\d^{n_{1,0}}$,
taking the residue of both sides of the above equality we get that
$\l'_1=\l_1$. Thus $T'_2=T_2$ and the statement is proved for $k=1$.

We apply induction on $k$. Suppose that $(t,Q_k)=1$. Then
$(t,Q_{k-1})=1$ and by the inductive assumption we have that
$q'_{k-1}=q_{k-1}$, $Q'_{k-1}=Q_{k-1}$ and $T'_j=T_j$ for all
$0<j\le k$. The coprime integers $p'_k$ and $q'_k$ satisfy the
following equality
$$
{p'_k}/{q'_k}=Q'_{k-1}(\tilde\n(T'_k)-q'_{k-1}\tilde\n(T'_{k-1}))=t
Q_{k-1}(\n^*(T_k)-q_{k-1}\n^*(T_{k-1}))={tp_k}/{q_k}.
$$
Since $(t,q_k)=1$ we get $p'_k=tp_k$ and $q'_k=q_k$. Since
$\tilde\n(T'_0)=1$, by the construction of jumping polynomials
$\{T'_i\}_{i\ge 0}$ in $S$ we get
$$
q_k\n^*(T_k)=
t^{-1}q'_k\tilde\n(T'_k)=t^{-1}\sum_{j=0}^{k-1}n'_{k,j}\tilde\n(T'_j)=
n'_{k,0}/t+\sum_{j=1}^{k-1}n'_{k,j}\n^*(T_j).
$$
On the other hand by the construction of jumping polynomials
$\{T_i\}_{i\ge 0}$ in $R$ we get that
$q_k\n^*(T_k)=\sum_{j=0}^{k-1}n_{k,j}\n^*(T_j)$ and this
representation is unique. Therefore, $n'_{k,0}=tn_{k,0}$ and
$n'_{k,j}=n_{k,j}$ for all $0<j<k$.

Finally, ${T'_k}^{q'_k}=T_k^{q_k}$ and
$\d^{n_{k,0}}\prod_{j=0}^{k-1}{T'_j}^{n'_{k,j}}=\prod_{j=0}^{k-1}T_j^{n_{k,j}}$.
Then, since the residue of $\d$ is 1 and
${T'_k}^{q'_k}(\prod_{j=0}^{k-1}{T'_j}^{n'_{k,j}})^{-1}=
T_k^{q_k}(\prod_{j=0}^{k-1}T_j^{n_{k,j}})^{-1}\d^{n_{k,0}}$, taking
the residue of both sides of the above equality we get that
$\l'_k=\l_k$. Thus $T'_{k+1}=T_{k+1}$ and the theorem is proved.

\end{proof}


\section{Behavior of jumping polynomials under
blow-ups}\label{behavior}

Throughout this section we work under the assumption that the value
group of $\n$ is a subgroup of $\mathbb{Q}$ and
$\trdeg_{\k}(V/{m_{V}})=0$.

We now introduce the notations used in the rest of the paper.

If $p$ and $q$ are positive integers such that $(p,q)=1$, the
Euclidian algorithm for finding the greatest common divisor of $p$
and $q$ can be described as follows:
\begin{align*}
r_0 &=f_1r_1+r_2\\
r_1 &=f_2r_2+r_3\\
&\vdots\\
r_{N-2} &=f_{N-1}r_{N-1}+1\\
r_{N-1} &=f_N\cdot 1,
\end{align*}
where $r_0=p$, $r_1=q$ and $r_1>r_2>\dots>r_{N-1}>r_N=1$. Denote by
$N=N(p,q)$ the number of divisions in the Euclidian algorithm for
$p$ and $q$, and by $f_1,\,f_2,\dots,f_N$ the coefficients in the
Euclidian algorithm for $p$ and $q$. Define $F_i=f_1+\dots+f_i$ and
$\e(p,q)=F_N=f_1+\dots+f_N$, $f_1(p,q)=f_1=\left[p/q \right]$.

Suppose that $R$ is a 2-dimensional regular local ring dominated by
$V$ and $E$ is a nonsingular irreducible curve on $\spec R$. Let
$$
R=R_0\ra R_1\ra R_2\ra\dots\ra R_i\ra\dots
$$
be the sequence of quadratic transforms along $\n$. We denote by
$\p_i$ the map $\spec R_i\ra\spec R$ and by $E_i$ the reduced simple
normal crossing divisor $\p_i^{-1}(E)_{red}$. We say that $R_i$ is
{\it free} if $E_i$ has exactly one irreducible component. For a
free ring $R_i$ and a regular parameter $u_i\in R_i$ we will say
that $u_i$ is an exceptional coordinate if $u_i$ is supported on
$E_i$. A system of parameters $(u_i,v_i)$ of a free ring $R_i$ is
called {\it permissible} if $u_i$ is an exceptional parameter. The
next lemma gives a description of the sequence of quadratic
transforms of $R$ along $\n$. See Section 6 of \cite{GHK} for the
proof.

\begin{lemma}\label{transformation}
Suppose that $R$ is a free ring and $(u,v)$ is a permissible system
of parameters in $R$ such that $\n(v)/\n(u)=p/q$ for some coprime
integers $p$ and $q$. Let $k=\e(p,q)$, $f_1=f_1(p,q)$ and let $a$
and $b$ be nonnegative integers such that $a\le p$, $b<q$, and
$aq-bp=1$. Then the sequence of quadratic transforms along $\n$
\begin{equation}\label{GSV}
R=R_0\ra R_1\ra\dots\ra R_{f_1}\ra R_{f_1+1}\ra\dots\ra R_{k-1}\ra
R_k
\end{equation}
has the following properties:
\begin{itemize}

\item[1)] $R_0,R_1,\dots,R_{f_1}$ and $R_k$ are free rings.

\medskip

\item[2)] Non-free rings appear in {\rm (\ref{GSV})} if and only
if $k>f_1$, that is, if $q\neq 1$.

\noindent In this case $R_{f_1+1},\dots,R_{k-1}$ are non-free.

\medskip

\item[3)] $R_k$ has a permissible system of coordinates
$$(X,Y)=\left(\dfrac{x^a}{y^b},\dfrac{y^q}{x^p}-c\right),$$ where
$c\in\k$ is the residue of $y^q/x^p$, and $\n^*(X)=\n^*(x)/q$.
Moreover, $x=X^q(Y+c)^b$, $y=X^p(Y+c)^a$.
\end{itemize}
\end{lemma}

\begin{definition}\cite[7.8]{C&P}
A permissible system of parameters $(u,v)$ of a free ring $R$ is
called {\it admissible} if $\n(v)$ is maximal among all regular
systems of parameters containing $u$.
\end{definition}

\begin{lemma}\label{permissible-admissible}
Suppose that $(u,v)$ is a permissible system of parameters of a free
ring $R$ and $\n(v)/\n(u)=p/q$ for some coprime integers $p$ and
$q$. Then $(u,v)$ are admissible if and only if $q\neq 1$.
\end{lemma}

\begin{proof}
Assume by contradiction that $q=1$. Then $\n(v)/\n(u)=p$. Denote by
$c$ the residue of $vu^{-p}$ and set $v'=v-cu^p$. Then $(u,v')$ are
permissible parameters with $\n(v')>\n(v)$. So $(u,v)$ are not
admissible.

Now assume by contradiction that $(u,v)$ are not admissible. Then
there exists $v'$ such that $(u,v')$ are permissible parameters and
$\n(v')>\n(v)$. By the Weierstrass Preparation theorem (Theorem 5,
Section 1, Chapter VII \cite{ZS}), we have that $\g v'=v-P(u)$,
where $\g$ is a unit, and $P\in \k[[u]]$ has order $n\geq 1$. Since
$\n(v')>\n(v)$, it follows that $n\n(u)=\n(v)$, and so
$\n(v)/\n(u)=n$. Thus $p=n$ and $q=1$, a contradiction.
\end{proof}

\begin{lemma}\label{samevalue}
Suppose that $(u,v)$ and $(\bar{u},\bar{v})$ are admissible systems
of parameters of a free ring $R$. Then $\n(u)=\n(\bar{u})$ and
$\n(v)=\n(\bar{v})$.
\end{lemma}

\begin{proof}
Since $u$ and $\bar{u}$ are supported on the same curve in $\spec R$
we have $\bar{u}=u\g$ for some unit $\g\in R$; in particular
$\n(u)=\n(\bar{u})$. The equality $\bar{u}=u\g$ also implies that
$(\bar{u},v)$ is a permissible system of parameters in $R$. Since
$(\bar{u},\bar{v})$ are admissible parameters we have
$\n(v)\le\n(\bar{v})$. Symmetrically, $\n(\bar{v})\le\n(v)$. Thus
$\n(v)=\n(\bar{v})$.
\end{proof}

\begin{remark}
If the value group $\G$ of $\n$ is a non-discrete subgroup of
$\mathbb{Q}$, it follows from \cite[7.7]{C&P} that an admissible
system of parameters of $R$ always exists. If $\G$ is a discrete
subgroup of $\mathbb{Q}$, after performing a sequence of quadratic
transforms along $\n$ we may assume that the value of the
exceptional parameter $u$ generates $\G$. It follows from Lemma
\ref{permissible-admissible} that $R$ does not have an admissible
system of parameters.
\end{remark}

We now fix regular parameters $(u,v)$ of $R$ and assume that $(u,v)$
is a permissible system of parameters in $R$ by setting $E$ to be
the curve on $\spec R$ defined by $u=0$. Recall the definition of
$\bar{p}_l$ and $\bar{q}_l$ given in Section \ref{jumpingpoly}. In
what follows, for all $l>0$ let $a_l$ and $b_l$ be nonnegative
integers such that $a_l\bar{q}_l-b_l\bar{p}_l=1$ and
$a_l\le\bar{p}_l$, $b_l<\bar{q}_l$. Let $\bar{k}_0=0$ and
$\bar{k}_l=\bar{k}_{l-1}+\e(\bar{p}_l,\bar{q}_l)$. Also set $k_0=0$
and $k_i=k_{i-1}+\e(p_i,q_i)$ if $i>0$.

The next theorem describes the images of independent jumping
polynomials under blowups of $R$ along $\n$.

\begin{theorem} \label{monoidalseq-lemma}
With notations as above, let $ R=R_0\ra R_1\ra\dots\ra
R_{\bar{k}_1-1}\ra R_{\bar{k}_1}\ra R_{\bar{k}_1+1}\ra\dots\ra
R_{\bar{k}_l}\ra\dots$ be the sequence of quadratic transforms along
$\n$. Assume that the value group of $\n$ is a non-discrete subgroup
of $\mathbb{Q}$. Then for all $l\ge 0$, $R_{\bar{k}_l}$ is free and
has an admissible system of parameters $(u_l,v_l)$ such that

\begin{itemize}

\item[1)]$u_l$ is an exceptional parameter and $\n(u_l)=1/\bar{Q}_l,$

\item[2)]$v_l=H_{l+1}/\prod_{j=0}^{l-1}H_j^{n_{i_l,i_j}}$ is the
strict transform of $H_{l+1}$ in $R_{\bar{k}_l}$ and
$\n(v_l)=(1/\bar{Q}_l)\cdot(\bar{p}_{l+1}/\bar{q}_{l+1}),$

\item[3)]for all $0 \le j \le l$ there exists a unit $\g_{j,l}\in
R_{\bar{k}_l}$ such that $H_j=u_l^{\bar{Q}_l\bar{\b}_j} \g_{j,l}$.
\end{itemize}
\end{theorem}

\begin{proof}
We first show that 3) implies that for all $1 \le m \le l+1$ and for
all $i_l<i'<i_{l+1}$ there exist units $\t_{m,l}$ and $\t'_{i',l}$
in $R_{\bar{k}_l}$ such that
$$
\prod_{j=0}^{m-1}H_j^{n_{i_m,i_j}}=
u_l^{\bar{Q}_l\bar{q}_m\bar{\b}_m}\t_{m,l}\quad\text{ and }\quad
\prod_{j=0}^l H_j^{n_{i',i_j}}= u_l^{\bar{Q}_l\b_{i'}}\t'_{i',l}.
$$

Indeed, the following lines of equalities hold
$$
\prod_{j=0}^{m-1}H_j^{n_{i_m,i_j}}=\prod_{j=0}^{m-1}(u_l^{\bar{Q}_l\bar{\b}_j}
\g_{j,l})^{n_{i_m,i_j}}=u_l^{\bar{Q}_l\sum_{j=0}^{m-1}n_{i_m,i_j}\bar{\b}_j}
\prod_{j=0}^{m-1}\g_{j,l}^{n_{i_m,i_j}}=u_l^{\bar{Q}_l\bar{q}_m\bar{\b}_m}\t_{m,l}
$$
and
$$
\prod_{j=0}^lH_j^{n_{i',i_j}}=
\prod_{j=0}^l(u_l^{\bar{Q}_l\bar{\b}_j}\g_{j,l})^{n_{i',i_j}}=
u_l^{\bar{Q}_l\sum_{j=0}^ln_{i',i_j}\bar{\b}_j}
\prod_{j=0}^l\g_{j,l}^{n_{i',i_j}}=
u_l^{\bar{Q}_lq_{i'}\b_{i'}}\t'_{i',l}=u_l^{\bar{Q}_l\b_{i'}}\t'_{i',l}.
$$

We will prove the theorem by induction on $l$. For $l=0$ the system
of parameters $(u_0, v_0)=(u, H_1)$ in $R$ clearly satisfies all the
conclusions. Assume now that the statement holds for $l-1$. Then by
Lemma \ref{transformation} the ring $R_{\bar{k}_l}$ is free and has
a permissible system of parameters $(u_l,z_l)$ such that
$$
u_l=\dfrac{u_{l-1}^{a_l}}{v_{l-1}^{b_l}},\quad
\n(u_l)=\dfrac{1}{\bar{Q}_l}\quad\text{and}\quad z_l=
\dfrac{v_{l-1}^{\bar{q}_l}}{u_{l-1}^{\bar{p}_l}}-c_l,
$$
where $c_l\in\k-\{0\}$ is the residue of
${v_{l-1}^{\bar{q}_l}}{u_{l-1}^{-\bar{p}_l}}$. It follows that $u_l$
satisfies conclusion 1). We will now show that $u_l$ also satisfies
conclusion 3).

If $j\le l-1$ then
$$
H_j=u_{l-1}^{\bar{Q}_{l-1}\bar{\b}_j}\g_{j,l-1}=
(u_l^{\bar{q}_l}(z_l+c_l)^{b_l})^{\bar{Q}_{l-1}\bar{\b}_j}\g_{j,l-1}=
u_l^{\bar{Q}_l\bar{\b}_j}\g_{j,l}
$$
and
\begin{align*}
H_l=v_{l-1}\prod_{j=0}^{l-2}H_j^{n_{i_{l-1},i_j}} & =v_{l-1}
u_{l-1}^{\bar{Q}_{l-1}\bar{q}_{l-1}\bar{\b}_{l-1}}\t_{l-1,l-1}=\\
 & =u_l^{\bar{p}_l}u_l^{\bar{Q}_l\bar{q}_{l-1}\bar{\b}_{l-1}}
(z_l+c_l)^{a_l+b_l\bar{Q}_{l-1}\bar{q}_{l-1}\bar{\b}_{l-1}}\t_{l-1,l-1}=
u_l^{\bar{Q}_l\bar{\b}_l}\g_{l,l},
\end{align*}
where $\g_{j,l}$ and $\g_{l,l}$ are units in $R_{\bar{k}_l}$.

Let us set $v_l=H_{l+1}/\prod_{j=0}^{l-1}H_j^{n_{i_l,i_j}}$. Then
$\n(v_l)=\bar{\b}_{l+1}-\bar{q_l}\bar{\b}_l=
(1/\bar{Q}_l)\cdot(\bar{p}_{l+1}/\bar{q}_{l+1})$.

It only remains to show that $(u_l,v_l)$ is an admissible system of
parameters in $R_{\bar{k}_l}$. To this end we will present $v_l$ in
terms of $u_l$ and $z_l$. We have that
$$
v_l=\dfrac{H_l^{\bar{q}_l}}{\prod_{j=0}^{l-1}H_j^{n_{i_l,i_j}}}
-\l_{i_l}\d_{i_l}-\sum_{i'=i_l+1}^{i_{l+1}-1}\l_{i'}\d_{i'}
\dfrac{\prod_{j=0}^lH_j^{n_{i',i_j}}}{\prod_{j=0}^{l-1}H_j^{n_{i_l,i_j}}}.
$$

Denote by $\t_l$ the unit $\t_{l-1,l-1}^{\bar{q}_l}\t_{l,l-1}^{-1}$
in $R_{\bar{k}_{l-1}}$ and denote by $t_l$ the residue of $\t_l$.
Then the following line of equalities holds
$$
\dfrac{H_l^{\bar{q}_l}}{\prod_{j=0}^{l-1}H_j^{n_{i_l,i_j}}}=
\dfrac{v_{l-1}^{\bar{q}_l}u_{l-1}^{\bar{Q}_l\bar{q}_{l-1}\bar{\b}_{l-1}}
\t_{l-1,l-1}^{\bar{q}_l}}
{u_{l-1}^{\bar{Q}_{l-1}\bar{q}_l\bar{\b}_l}\t_{l,l-1}}=
\dfrac{v_{l-1}^{\bar{q}_l}}{u_{l-1}^{\bar{p}_l}}\t_l=(z_l+c_l)\t_l=
z_l\t_l+c_lt_l+u_lw_l,
$$
where $w_l\in R_{\bar{k}_l}$ is such that $u_lw_l=c_l(\t_l-t_l)$,
and the existence of $w_l$ is guaranteed by the inclusion
$(\t_l-t_l)\in m_{R_{\bar{k}_{l-1}}}R_{\bar{k}_{l-1}}\subset
u_lR_{\bar{k}_l}$. We also notice that taking the residues of both
sides of this equality gives $\l_{i_l}=c_lt_l$. Furthermore, since
$(\d_{i_l}-1)\in m_RR\subset u_lR_{\bar{k}_l}$ we have
$\l_{i_l}\d_{i_l}=\l_{i_l}+u_lw'_l$ for some $w'_l\in
R_{\bar{k}_l}$.

For $i_l< i'<i_{l+1}$ denote by $P_{i'}$ the positive integer
$\bar{Q}_l(\b_{i'}-\bar{q_l}\bar{\b}_l)=p_{i_l+1}+p_{i_l+2}+\dots+p_{i'}$.
Then
\begin{align*}
\sum_{i'=i_l+1}^{i_{l+1}-1}\l_{i'}\d_{i'}
\dfrac{\prod_{j=0}^lH_j^{n_{i',i_j}}}{\prod_{j=0}^{l-1}H_j^{n_{i_l,i_j}}}
&=
\sum_{i'=i_l+1}^{i_{l+1}-1}\l_{i'}\d_{i'}\dfrac{u_l^{\bar{Q}_l\b_{i'}}\t'_{i',l}}
{u_l^{\bar{Q}_l\bar{q}_l\bar{\b}_l}\t_{l,l}} =
\sum_{i'=i_l+1}^{i_{l+1}-1}\l_{i'}\d_{i'}\t'_{i',l}\t_{l,l}^{-1}u_l^{P_i'}\\
&= u_l
\sum_{i'=i_l+1}^{i_{l+1}-1}\l_{i'}\d_{i'}\t'_{i',l}\t_{l,l}^{-1}u_l^{P_i'-1}=
u_lw''_l,
\end{align*}
where $w''_l\in R_{\bar{k}_l}$. In particular, $w''_l=0$ if
$i_{l+1}=i_l+1$.

Combining all the above we get
$v_l=z_l\t_l+u_lw_l-u_lw'_l-u_lw''_l$. Thus $(u_l,v_l)$ form a
system of regular parameters in $R_{\bar{k}_l}$. Moreover,
$(u_l,v_l)$ is an admissible system by Lemma
\ref{permissible-admissible}, since
$\n(v_l)/\n(u_l)=\bar{p}_{l+1}/\bar{q}_{l+1}$ and $\bar{q}_{l+1}\neq
1$. This completes the proof of the theorem.
\end{proof}

Using the same line of arguments with $H_l$ replaced by $T_i$ and
$\bar{k}_l$ replaced by $k_i$, we obtain the following property of
the sequence of jumping polynomials in $R$.

\begin{remark}(see also \cite[7.5]{GHK}).\label{shortchunks}
Let $ R=R_0\ra R_1\ra\dots\ra R_{k_1-1}\ra R_{k_1}\ra
R_{k_1+1}\ra\dots\ra R_{k_i}\ra\dots$ be the sequence of quadratic
transforms along $\n$. Then for all $i>0$, $R_{k_i}$ is free and has
a system of regular parameters $(u_i,v_i)$ such that

\begin{itemize}

\item[1)]$u_i$ is an exceptional parameter and $\n(u_i)=1/Q_i,$

\item[2)]$v_i=T_{i+1}/\prod_{j=0}^{i-1}T_j^{n_{i,j}}$ is the strict
transform of $T_{i+1}$ in $R_{k_i}$ and
$\n(v_i)=(1/Q_i)\cdot(p_{i+1}/q_{i+1}),$

\item[3)]for all $0 \le j \le i$ there exists a unit $\g_{j,i}\in R_{k_i}$ such
that $T_j = u_i^{Q_i\b_j} \g_{j,i}$.
\end{itemize}
\end{remark}

\begin{remark}\label{admissibleandchunks}
For every $l\geq 0$ we have that $\bar{k}_l=k_{i_l}$. This is
trivial for $l=0$. By induction on $l$, assume that
$\bar{k}_{l-1}=k_{i_{l-1}}$. We have that
$
\e(\bar{p}_l,\bar{q}_l)=\e((p_{i_{l-1}+1}+\dots+p_{i_l-1})q_{i_l}+p_{i_l},q_{i_l})=p_{i_{l-1}+1}+\dots+p_{i_l-1}+
\e(p_{i_l},q_{i_l})=\e(p_{i_{l-1}+1},q_{i_{l-1}+1})+\dots+\e(p_{i_l-1},q_{i_l-1})+
\e(p_{i_l},q_{i_l}),
$
since $q_{i_{l-1}+1}=\dots=q_{i_l-1}=1$. Therefore
$\bar{k}_l=\bar{k}_{l-1}+\e(\bar{p}_l,\bar{q}_l)=k_{i_{l-1}}+\e(p_{i_{l-1}+1},q_{i_{l-1}+1})+\dots
+\e(p_{i_l-1},q_{i_l-1})+ \e(p_{i_l},q_{i_l})=k_{i_l}$.
\end{remark}

\begin{remark}
Notice that $(u_i,v_i)$ of Remark \ref{shortchunks} are not in
general admissible parameters of $R_{k_i}$. If the index $k_i$
corresponds to an admissible choice of parameters then $i+1=i_l$ for
some $l$, that is, $k_i=k_{i_l-1}=\bar{k_l}-\e(p_{i_l}, q_{i_l})$.
\end{remark}


\section{Strong Monomialization}\label{strongmonomialization}

In this section we recall definitions and results from Section 7 of
\cite{C&P} that will be needed in this paper.

Let $\k$ be an algebraically closed field of characteristic $\bp>0$,
and let $K^*/K$ be a finite and separable extension of algebraic
function fields of transcendence degree two over $\k$. Let $\n^*$ be
a $\k$-valuation of $K^*$ with valuation ring $V^*$ and value group
$\G^*$. Let $\n$ be the restriction of $\n^*$ to $K$ with valuation
ring $V$ and value group $\G$. Consider an extension of algebraic
regular local rings $R \subset S$ where $R$ has quotient field $K$,
$S$ has quotient field $K^*$, $R$ is dominated by $S$ and $S$ is
dominated by $V^*$. We assume that $\G^*$ and $\G$ are non-discrete
subgroups of $\mathbb{Q}$. In particular we have that
$\trdeg_{\k}(V^*/{m_{V^*}})=0$, and so $V^*/{m_{V^*}}\simeq\k$,
since $\k$ is algebraically closed. Let $S=S_0\ra S_1\ra
S_2\ra\dots\ra S_s\ra\dots$ be the quadratic sequence along $\n^*$,
and let $R=R_0\ra R_1\ra R_2\ra\dots\ra R_r\ra\dots$ be the
quadratic sequence along $\n$. For $i\geq 1$ we denote the reduced
exceptional locus of $\spec R_i\ra\spec R$ by $E_i$, and the reduced
exceptional locus of $\spec S_i\ra\spec S$ by $F_i$.

\begin{definition} Given a pair $(r,s)$ of positive integers, the
pair $(R_r,S_s)$ is said to be {\it prepared} if the following
properties hold:
\begin{itemize}
\item[i)] $S_s$ dominates $R_r$.
\item[ii)] $R_r$ and $S_s$ are free.
\item[iii)] The critical locus of $\spec S_s\ra\spec R_r$ is contained in
$F_s$.
\item[iv)] We have $u= x^t\d$, where $u$ (resp. $x$) is a regular
parameter of $R_r$ (resp. $S_s$) whose support is $E_r$ (resp.
$F_s$), and $\d$ is a unit in $S_s$.
\end{itemize}
\end{definition}

It is shown in \cite[7.6]{C&P} that given a prepared pair
$(R_r,S_s)$, any pair $(R_{r'},S_{s'})$ with $r'\geq r$, $s'\geq s$,
and such that both $R_{r'}$ and $S_{s'}$ are free and $S_{s'}$
dominates $R_{r'}$, is also prepared.

Let $(R_r,S_s)$ be prepared. Recall that a regular system of
parameters (r.s.p.) $(u,v)$ of $R_r$ is said to be admissible if the
support of $u$ is equal to $E_r$ and if $\n(v)$ is maximal among all
such r.s.p. containing $u$.

Now we briefly recall the algorithm described in Section 7.4 of
\cite{C&P}. We fix a prepared pair $(R,S)=:(R_{r_0},S_{s_0})$ such
that $m_RS$ is not a principal ideal. By induction on $n\geq 0$, we
associate with a given prepared pair $(R_{r_n},S_{s_n})$ such that
$m_{R_{r_n}}S_{s_n}$ is not a principal ideal, a new prepared pair
$(R_{r_{n+1}},S_{s_{n+1}})$, with $r_{n+1}>r_n$, $s_{n+1}>s_n$, and
such that $m_{R_{r_{n+1}}}S_{s_{n+1}}$ is not a principal ideal.

Let $(u_{r_n},v_{r_n})$ be an admissible r.s.p. of $R_{r_n}$. Let
$(x_{s_n},y_{s_n})$ be a r.s.p. of $S_{s_n}$ such that the support
of $x_{s_n}$ is $F_{s_n}$. Write
\begin{equation}
\begin{array}{ll}
u_{r_n} &=  x_{s_n}^{t_n}\d_n  \\
v_{r_n} &=  x_{s_n}^{\beta_n}f_n \\
\end{array}
\end{equation}
where $\d_n$ is a unit in $S_{s_n}$, and $x_{s_n}$ does not divide
$f_n$. Notice that $f_n$ is not a unit, since $m_{R_{r_n}}S_{s_n}$
is not a principal ideal. Among all $s\geq s_n$, there is a least
integer $s_{n+1}$ such that $S_{s_{n+1}}$ is free and the strict
transform of div$(f_n)$ in $S_{s_{n+1}}$ is empty. It follows that
$s_{n+1}>s_n$. By construction, $m_{R_{r_n}}S_{s_{n+1}}$ is a
principal ideal. The nonempty set of integers $r>r_n$ such that
$S_{s_{n+1}}$ dominates $R_r$ has a maximal element denoted by
$r_{n+1}$. This completes the definition of the pair
$(R_{r_{n+1}},S_{s_{n+1}})$. It is shown in \cite[7.18]{C&P} that
the algorithm is well defined, that is, the pair $(r_{n+1},
s_{n+1})$ does not depend on the choice of an admissible r.s.p.
$(u_{r_n},v_{r_n})$ of $R_{r_n}$. Moreover,
$(R_{r_{n+1}},S_{s_{n+1}})$ is prepared.

We are now ready to state Cutkosky and Piltant's Strong
Monomialization theorem for defectless extensions.

\begin{theorem} {\rm (Strong Monomialization
\cite[7.35]{C&P})}\label{strongmonom} In the above set-up and
notations, assume that $V^*/V$ is defectless. The inclusion
$R_{r_n}\subset S_{s_n}$ is given for $n>>0$ by
\begin{equation}
\begin{array}{ll}
u_{r_n} &=   x_{s_n}^t\d_n \\
v_{r_n} &=  y_{s_n}\\
\end{array}
\end{equation}
where $t$ is a positive integer, $\d_n$ is a unit in $S_{s_n}$, and
$(x_{s_n}, y_{s_n})$ is an admissible r.s.p. of $S_{s_n}$.
\end{theorem}

In particular the theorem shows that the equation defining the
inclusion $R_{r_n}\subset S_{s_n}$ gets a stable form. We recall
that Strong Monomialization may not hold if the extension $V^*/V$
has a defect \cite[7.38]{C&P}. It is not known if the weaker form of
the monomialization theorem \cite[4.1]{C&P} holds in this case.

\section{Quadratic transforms}\label{qt}

In this section we discuss several preparatory results that we will
need in the proof of the main theorem.

Throughout this section let $R \subset S$ be an extension of two
dimensional algebraic regular local rings with quotient fields $K$
and $K^*$ respectively, such that $V^*$ dominates $S$ and $S$
dominates $R$. We further assume that $\G^*$ and $\G$ are
non-discrete subgroups of $\mathbb{Q}$. Suppose that $R$ has regular
parameters $(u,v)$ and $S$ has regular parameters $(x,y)$. Let $p$
and $q$ be positive coprime integers such that $\n(v)/\n(u)=p/q$ and
let $k=\e(p,q)$. Let $p'$ and $q'$ be positive coprime integers such
that $\n^*(y)/\n^*(x)=p'/q'$ and let $k'=\e(p',q')$.

\begin{remark}\label{admissibleforRandS}
Suppose that the inclusion $R\subset S$ is given by
\begin{equation}
\begin{array}{ll}
u &=  x^t \d  \\
v &=  y,\\
\end{array}
\end{equation}
where $t$ is a positive integer and $\d$ is a unit in $S$. We have
that $\n(v)/\n(u)=\n^*(y)/t\n^*(x)=p'/tq'$. Suppose that $(x,y)$ are
admissible parameters. Then by Lemma \ref{permissible-admissible}
$(u,v)$ are admissible, since $q'\neq 1$ implies $q\neq 1$.
Viceversa, if $(u,v)$ are admissible and $t$ divides $p'$, then
$(x,y)$ are admissible, since in this case $q'=q$.
\end{remark}

\begin{remark}\label{chunk}
Suppose that the inclusion $R\subset S$ is given by
\begin{equation}
\begin{array}{ll}
u &=  x^t \d  \\
v &=  y,\\
\end{array}
\end{equation}
where $t$ is a positive integer and $\d$ is a unit in $S$. Let $g$
be the greatest common divisor of $t$ and $p'$. Recall that
$\n(v)/\n(u)=\n^*(y)/t\n^*(x)=p'/tq'$. Writing $t=g \tilde t$ and
$p'=g\tilde p$, where $(\tilde t,\tilde p)=1$, gives $p=\tilde p$
and $q= q'\tilde t$. After possibly multiplying $u$ by a constant we
may assume that $\d=1+w$ for some $w\in m_S$. Let $a, b, a', b'$ be
nonnegative integers such that $a\le p$, $a'\le p'$, $b<q$, $b'<q'$
and $aq-bp=1$, $a'q'-b'p'=1$.

By Lemma \ref{transformation} applied to $S$ and $R$ respectively,
we get that $S_{k'}$ has a permissible system of parameters
$(X,Y')=\left(x^{a'}/y^{b'},y^{q'}/x^{p'}-c'\right)$, where
$c'\in\k$ is the residue of $y^{q'}/x^{p'}$, and $R_{k}$ has a
permissible system of parameters
$(U,V)=\left(u^{a}/v^{b},v^{q}/u^{p}- c\right)$, where $c\in\k$ is
the residue of $v^{q}/u^{p}$. Moreover, $x=X^{q'}(Y'+c')^{b'}$,
$y=X^{p'}(Y'+c')^{a'}$ and $u=U^{q}(V+c)^{b}$, $v=U^{p}(V+c)^{a}$.

Now
$$
U=\frac{u^{a}}{v^{b}}=\frac{x^{t a}\d^{a}}{y^{
b}}=\frac{[X^{q'}(Y'+c')^{b'}]^{t a}\d^{a}}{[X^{p'}(Y'+c')^{a'}]^{
b}}=\frac{X^{q't a}}{X^{p' b}}\d^{a}(Y'+c')^{b'ta- a'b}= X^g\D,
$$
where $\D=\d^{a}(Y'+c')^{b'ta- a'b}$ is a unit in $S_{k'}$. Notice
that the last equality holds since $q'ta-p'b=q'g\tilde ta-g\tilde
pb=g(q'\tilde ta-\tilde p b)=g( q a-p b)=g$.

Furthermore notice that $$\dfrac{v^{q}}{u^{p}}=\dfrac{y^{q}}{x^{t
p}}\d^{-p}=\dfrac{y^{q'\tilde t}}{x^{g\tilde tp}}\d^{-p}=\left
(\dfrac{y^{q'}}{x^{p'}}\right)^{\tilde t}\d^{-p}=\left
(\dfrac{y^{q'}}{x^{p'}}\right)^{\tilde t}(1+w)^{-p}.$$ Therefore $
c=(c')^{\tilde t}$, and
$$
V=\frac {v^{q}}{u^{p}}-c=\left (\frac{y^{q'}}{x^{p'}}\right)^{\tilde
t}(1+w)^{-p}-(c')^{\tilde t}=\left
(\frac{y^{q'}}{x^{p'}}\right)^{\tilde t}-(c')^{\tilde t}+W$$

where $W\in wS\subset m_S$. Since $m_S\subset (X)S_{k'}$ we have
that $W=XZ$ for some $Z\in S_{k'}$. Write $\tilde t=\bp^nt'$, where
$n\geq 0$ and $\bp$ does not divide $t'$.

Then $$V=\left[ \left
(\frac{y^{q'}}{x^{p'}}\right)^{t'}-(c')^{t'}\right
]^{\bp^n}+XZ=\left (\frac{y^{q'}}{x^{p'}}-c'\right)^{\bp^n}
\gamma^{\bp^n}+XZ=(Y')^{\bp^n}\gamma_1+XZ,$$

where $$\gamma=\dfrac {\left
(\frac{y^{q'}}{x^{p'}}\right)^{t'}-(c')^{t'}}{\left
(\frac{y^{q'}}{x^{p'}}-c'\right)}$$ and $\gamma_1=\gamma^{\bp^n}$ is
a unit in $S_{k'}$. In particular $S_{k'}$ dominates $R_k$.

\end{remark}

\begin{remark}\label{goodchunk}
In the set-up of Remark \ref{chunk} we will be interested in the
case when $(t,q)=1$, or equivalently $t$ divides $p'$. In this case
$g=t$, $\tilde t=1$, $t'=1$ and $n=0$. In particular we have that
$V=Y'\gamma_1+XZ$. Setting $Y=Y'\gamma_1+XZ$, we have that $(X,Y)$
is a permissible system of parameters of $S_{k'}$, and the inclusion
$R_{k}\subset S_{k'}$ is given by
\begin{equation}
\begin{array}{ll}
U &=   X^t \Delta \\
V &=  Y.\\
\end{array}
\end{equation}
\end{remark}

\begin{remark}\label{preparedpairs}
 Suppose that the inclusion $R=R_{r_0}\subset S=S_{s_0}$ is given by
\begin{equation}
\begin{array}{ll}
u &=   x^t \d \\
v &=  y,\\
\end{array}
\end{equation}
where $t$ is a positive integer and $\d$ is a unit in $S$. Suppose
also that $(u,v)$ are admissible parameters of $R$. We claim that
$(R_{k},S_{k'})=(R_{r_1},S_{s_1})$. By definition (see Section
\ref{strongmonomialization}), $S_{s_1}$ is the first free ring in
the quadratic sequence for $S$ such that the strict transform of $y$
in such ring is empty. The sequence $S=S_0\ra S_1\ra\dots\ra S_{k'}$
has been explicitly described in Section 6 of \cite{GHK} (see also
Lemma \ref{transformation}). It follows from this description that
$S_{s_1}=S_{k'}$. Furthermore, by Remark \ref{chunk} $S_{k'}$
dominates $R_{k}$, and $R_{k}$ is the biggest free ring in the
quadratic sequence for $R$ with this property. So $R_{k}=R_{r_1}$.
\end{remark}


As in Section \ref{jumpingpoly}, for all $i>0$ let $(p_i,q_i)$ be
the pair of coprime integers defined in the construction of jumping
polynomials $\{T_i\}_{i\ge 0}$ in $R$. Let $\{T_{i_l}\}_{l\geq 0}$
be the sequence of independent jumping polynomials in $R$. For $l>0$
let $\bar{q}_l=q_{i_l}$,
$\bar{p}_l=(p_{i_{l-1}+1}+\dots+p_{i_l-1})\bar{q}_l+p_{i_l}$. Let
$\bar{k}_0=0$, $\bar{k}_l=\bar{k}_{l-1}+\e(\bar{p}_l,\bar{q}_l)$ if
$l>0$. Let $k_0=0$ and $k_i=k_{i-1}+\e(p_i,q_i)$ if $i>0$. For all
$i>0$ let $(p'_i,q'_i)$ be the pair of coprime integers defined in
the construction of jumping polynomials $\{T'_i\}_{i\ge 0}$ in $S$.
Let $\{T'_{j_l}\}_{l\geq 0}$ be the sequence of independent jumping
polynomials in $S$. For $l>0$ let $\bar{q}'_l=q'_{j_l}$,
$\bar{p}'_l=(p'_{j_{l-1}+1}+\dots+p'_{j_l-1})\bar{q}'_l+p'_{j_l}$.
Let $\bar{k}'_0=0$,
$\bar{k}'_l=\bar{k}'_{l-1}+\e(\bar{p}'_l,\bar{q}'_l)$ if $l>0$.
Finally, let $k'_0=0$, and $k'_i=k'_{i-1}+\e(p'_i,q'_i)$ if $i>0$.

\begin{lemma}\label{admissibleparameters}
Assume that the equation defining the inclusion $R_{r_l}\subset
S_{s_l}$ for $l\geq 0$ has the stable form of Theorem
\ref{strongmonom}. Then $R_{r_l}=R_{\bar{k}_l}=R_{k_{i_l}}$, and
$S_{s_l}=S_{\bar{k}'_l}=S_{k'_{j_l}}$. In particular, if
$(u_{r_l},v_{r_l})$ is an admissible system of parameters of
$R_{r_l}$ we have that $\n(v_{r_l})/\n(u_{r_l})=\bar {p}_{l+1}/\bar
{q}_{l+1}$. Similarly, if $(x_{s_l},y_{s_l})$ is an admissible
system of parameters of $S_{s_l}$ we have that
$\n^*(y_{s_l})/\n^*(x_{s_l})=\bar {p}'_{l+1}/\bar {q}'_{l+1}.$
\end{lemma}

\begin{proof}
It suffices to show that for every $l\geq 0$ we have that
$R_{r_l}=R_{\bar{k}_l}$ and $S_{s_l}=S_{\bar{k}'_l}$. Then the other
claims follow from Remark \ref{admissibleandchunks}, Theorem
\ref{monoidalseq-lemma} and Lemma \ref{samevalue}.

We apply induction on $l$. The case $l=0$ is trivial since by
definition $R=R_{r_0}=R_{\bar {k}_0}$ and $S=S_{s_0}=S_{\bar
{k}'_0}$. Now assume that $R_{r_{l-1}}=R_{\bar{k}_{l-1}}$ and
$S_{s_{l-1}}=S_{\bar{k}'_{l-1}}$, that is, $r_{l-1}=\bar{k}_{l-1}$
and $s_{l-1}=\bar{k}'_{l-1}$. By Remark \ref{preparedpairs} applied
to the inclusion $R_{r_{l-1}}\subset S_{s_{l-1}}$, we have that
$r_l=r_{l-1}+\e(\bar{p}_l,\bar{q}_l)=\bar{k}_{l-1}+\e(\bar{p}_l,\bar{q}_l)=\bar{k}_l$
and
$s_l=s_{l-1}+\e(\bar{p}'_l,\bar{q}'_l)=\bar{k}'_{l-1}+\e(\bar{p}'_l,\bar{q}'_l)=\bar{k}'_l$.
\end{proof}


\section{Monomialization of generating sequences}

The goal of this section is to prove the following theorem.

\begin{theorem}\label{monomialization}
Let $\k$ be an algebraically closed field of characteristic $\bp>0$,
and let $K^*/K$ be a finite separable extension of algebraic
function fields of transcendence degree $2$ over $\k$. Let $\n^*$ be
a $\k$-valuation of $K^*$, with valuation ring $V^*$ and value group
$\G^*$, and let $\n$ be the restriction of $\n^*$ to $K$, with
valuation ring $V$ and value group $\G$. Assume that $V^*/V$ is
defectless. Suppose that $R \subset S$ is an extension of algebraic
regular local rings with quotient fields $K$ and $K^*$ respectively,
such that $V^*$ dominates $S$ and $S$ dominates $R$. Then there
exist sequences of quadratic transforms $R \to \bar{R}$ and $S \to
\bar{S}$ along $\n^*$ such that $\bar{S}$ dominates $\bar{R}$ and
the map between generating sequences of $\n$ and $\n^*$ in $\bar{R}$
and $\bar{S}$ respectively, has a toroidal structure.
\end{theorem}

\begin{proof}

By the discussion of Section \ref{easycases} we only need to
consider the case when $\G^*$ is a subgroup of $\mathbb{Q}$ and
$\trdeg_{\k}(V^*/{m_{V^*}})=0$.

 By
\cite[7.3]{C&P} if $\G^*$ and $\G$ are discrete, and by Theorem
\ref{strongmonom} if $\G^*$ and $\G$ are non-discrete, we may assume
that $R$ has a regular system of parameters $(u,v)$, $S$ has a
regular system of parameters $(x,y)$ such the inclusion $R\subset S$
is given by
\begin{equation}\label{monomialequation}
\begin{array}{ll}
u &=   x^t \d \\
v &=  y,\\
\end{array}
\end{equation}
where $t$ is a positive integer and $\d$ is a unit in $S$. After
possibly multiplying $u$ by a constant we may assume that $\d=1+w$
for some $w\in m_S$. If $t=1$ then the conclusion of the theorem is
trivial, so assume that $t>1$. We may also assume that $\G$ is
normalized so that $\n(u)=1$.

As in Section \ref{jumpingpoly}, let $\{T_i\}_{i\ge 0}$ be the
sequence of jumping polynomials in $R$ corresponding to $(u,v)$ and
the trivial sequence of units $\{1\}_{i>0}$. For all $i>0$ let the
coprime integers $p_i$ and $q_i$ be defined as in the construction
of $\{T_i\}_{i\ge 0}$. Let $\{n_{i,0}\}_{i>0}$ be the powers defined
in the construction of $\{T_i\}_{i\ge 0}$. For $l>0$ let the coprime
integers $\bar{p}_l$ and $\bar{q}_l$ be defined as in the
construction of the subsequence $\{T_{i_l}\}_{l\geq 0}=\{H_l\}_{l\ge
0}$ of independent jumping polynomials.

Similarly, let $\{T'_i\}_{i\ge 0}$ be the sequence of jumping
polynomials in $S$ corresponding to $(x,y)$ and the sequence of
units $\{\d^{n_{i,0}}\}_{i>0}$. For all $i>0$ let the coprime
integers $p'_i$ and $q'_i$ be defined as in the construction of
$\{T'_i\}_{i\ge 0}$. For $l>0$ let the coprime integers $\bar{p}'_l$
and $\bar{q}'_l$ be defined as in the construction of the
subsequence $\{H'_l\}_{l\ge 0}$ of independent jumping polynomials.

First, let us assume that $\G^*$ and $\G$ are discrete subgroups of
$\mathbb{Q}$. After performing a sequence of quadratic transforms
along $\n$ and normalizing $\G$ we may further assume that $\n(u)=1$
generates $\G$. Since $q_i=1$ for all $i>0$, by Theorem
\ref{relationship} we have that $T'_i=T_i$ for all $i>0$. Then by
Theorem \ref{nondiscretesequence} $\{T_i\}_{i\geq 0}$ is a
generating sequence in $R$ and $\{T'_i\}_{i\geq
0}=\{x,\{T_i\}_{i>0}\}$ is a generating sequence in $S$. Notice also
that $\G^*$ is generated by the set of values $\{\n^*(T'_i)\}_{i\geq
0}$, and that for any $i>0$ the value $\n^*(T'_i)=\n(T_i)$ is an
integer. Thus $\n^*(x)=\n(u)/t=1/t$ generates $\G^*$. This concludes
the proof of the theorem in the discrete case.

Now let us assume that $\G^*$ and $\G$ are non-discrete subgroups of
$\mathbb{Q}$. By Theorem \ref{strongmonom} we may assume that
equation (\ref{monomialequation}) is stable and that
$(u,v)=(u_{r_0},v_{r_0})$, $(x,y)=(x_{s_0},y_{s_0})$ are admissible
parameters of $R=R_{r_0}$, $S=S_{s_0}$ respectively.

First, let us assume that $(t,Q_k)=1$ for all $k>0$. By Theorem
\ref{relationship} we have that $T'_i=T_i$ for all $i>0$. Since
$q'_i=q_i$ for all $i>0$, we have that $T'_i$ is an independent
jumping polynomial in $S$ if and only if $T_i$ is an independent
jumping polynomial in $R$. If follows that $\{H'_l\}_{l\ge 0}=
\{x,\{H_l\}_{l>0}\}$.

By Theorem \ref{nondiscreteminsequence} $\{H_l\}_{l\ge 0}$ is a
generating sequence in $R$ and $\{x,\{H_l\}_{l> 0}\}$ is a
generating sequence in $S$. Moreover, $\{x,\{H_l\}_{l> 0}\}$ is a
minimal generating sequence of $\n^*$ since $\bar{p}_1'\geq
p'_1=tp_1>1$, since $t>1$. The conclusion of the theorem follows.

Otherwise let $M$ be the integer such that $(t,Q_k)=1$ for all
$0\leq k<M$ and $(t, q_M)\neq 1$.
Since $(t,q_{M})\neq 1$ it follows that $q_{M}\neq 1$, and so
$M=i_l$ for some $l$. The inclusion $R_{r_{l-1}}\subset S_{s_{l-1}}$
is given by the stable monomial form
\begin{equation}
\begin{array}{ll}
u_{r_{l-1}} &=   x_{s_{l-1}}^{t} \d_{l-1}\\
v_{r_{l-1}} &=  y_{s_{l-1}}\\
\end{array}
\end{equation}
where $\d_{l-1}$ is a unit in $S_{s_{l-1}}$,
$(x_{s_{l-1}},y_{s_{l-1}})$ are admissible parameters of
$S_{s_{l-1}}$, and $(u_{r_{l-1}},v_{r_{l-1}})$ are admissible
parameters of $R_{r_{l-1}}$.

By Lemma \ref{admissibleparameters} we have that
$\n(v_{r_{l-1}})/\n(u_{r_{l-1}})=\bar p_l/\bar q_l$, and
$\n^*(y_{s_{l-1}})/\n^*(x_{s_{l-1}})=\bar p'_l/\bar q'_l$. Recall
that $\bar q_l=q_{i_l}=q_M$. We have that $(t, \bar q_l)\neq 1$, or
equivalently $t$ does not divide $\bar p'_l$.
Now we apply Remark \ref{chunk} and Remark \ref{preparedpairs} with
$R$ and $S$ replaced by $R_{r_{l-1}}$ and $S_{s_{l-1}}$
respectively.

It follows that $R_{r_{l}}$ has permissible regular parameters
$(u_{s_{l}},v_{s_{l}})$ and $S_{s_{l}}$ has permissible regular
parameters $(x_{s_{l}},y_{s_{l}})$ such that $u_{s_{l}}
=x_{s_{l}}^g\D$, where $\D$ is a unit in $S_{s_{l}}$ and $g=(\bar
p'_l ,t)<t$. Hence we obtain a contradiction to the stable form of
equation (\ref{monomialequation}).

We then conclude that $(t,Q_k)=1$ for all $k>0$.
\end{proof}

\begin{remark}
In the proof of Theorem \ref{monomialization} we have
$\{H_l\}_{l\geq 0}$ is a minimal generating sequence of $\n$ in $R$
if $\bar p_1\neq 1$, or equivalently, ${\bar p}'_1\neq t$. Otherwise
$\{H_l\}_{l>0}$ is a minimal generating sequence of $\n$ in $R$.
\end{remark}

Recall that the integers $k_i$ are defined as $k_0=0$ and
$k_{i}=k_{i-1}+\e(p_i,q_i)$ if $i>0$. Similarly $k'_0=0$ and
$k'_i=k'_{i-1}+\e(p'_i,q'_i)$ for all $i>0$.

\begin{corollary}\label{claim} In the set-up of Theorem \ref{monomialization}, assume that the inclusion $R\subset S$ is given
by the stable form (\ref{monomialequation}). The sequences of
quadratic transforms
$$
R=R_0\ra R_{k_1}\ra\dots\ra R_{k_{i-1}}\ra R_{k_i}\ra\dots$$

and $$S=S_0\ra S_{k'_1}\ra\dots\ra S_{k'_{i-1}}\ra S_{k'_i}\ra\dots
$$
have the following properties: for all $i\geq 0$ the rings $R_{k_i}$
and $S_{k'_i}$ are free, there exist permissible systems of
parameters $(u_i,v_i)$ in $R_{k_i}$ and $(x_i,y_i)$ in $S_{k'_i}$
and a unit $\d_i\in S_{k'_i}$ such that
\begin{equation}
\begin{array}{ll}
u_i &=   x_i^t \d_i \\
v_i &=  y_i,\\
\end{array}
\end{equation}
and $\n^*(y_i)/\n^*(x_i)=p'_{i+1}/q'_{i+1}$.
\end{corollary}
\begin{proof} The proof of Theorem \ref{monomialization}
shows that for all $i\geq 1$ we have $(t,q_i)=1$, $q'_i=q_i$, and
$p'_i=tp_i$. The conclusion is trivial for $i=0$. Assume that $i>0$
and that the statement holds for $i-1$. We apply Remark
\ref{goodchunk} to $R_{k_{i-1}}\subset S_{k'_{i-1}}$. Notice that
$\n^*(y_{i-1})/\n^*(x_{i-1})=p'_i/q'_i$ and
$\n(v_{i-1})/\n(u_{i-1})=p'_i/tq'_i=p_i/q_i$, and therefore
$k=\e(p_i,q_i)$ and $k'=\e(p'_i,q'_i)$. Thus $R_{k_i}$ and
$S_{k'_i}$ are free rings and there exist permissible systems of
regular parameters $(u_i,w_i)$ in $R_{k_i}$ and $(x_i,z_i)$ in
$S_{k'_i}$ such that $u_i=x_i^t\d_i$ and $w_i=z_i$ for some unit
$\d_i\in S_{k'_i}$.

Now by Remark \ref{shortchunks} we get that $R_{k_i}$ has a system
of regular parameters $(h_i,v_i)$ such that $h_i$ is an exceptional
parameter, $\n(h_i)=1/Q_i$ and $ \n(v_i)=
(1/Q_i)\cdot(p_{i+1}/q_{i+1})$. Since $u_i$ is also an exceptional
parameter in $R_{k_i}$ we have $u_i=h_i\g$ for some unit $\g\in
R_{k_i}$. Therefore $(u_i,v_i)$ form a permissible system of
parameters in $R_{k_i}$ and $\n(u_i)=1/Q_i$. Notice also that
$v_i=\a u_i+\b w_i$, where $\a,\b\in R_{k_i}$. Moreover, $\b$ is a
unit in $R_{k_i}$, since the image of $v_i$ is a regular parameter
in $R_{k_i}/(u_i)R_{k_i}$. This implies that $v_i=\a x_i^t\d_i+\b
z_i$ is also a regular parameter in $S_{k'_i}$ and $(x_i,v_i)$ form
a permissible system of parameters in $S_{k'_i}$. We set $y_i=v_i$
and observe that
$$\n^*(x_i)=\dfrac{1}{t}\n(u_i)=\dfrac{1}{tQ_i}\ \   {\rm and}\ \
\n^*(y_i)=\n(v_i)=\dfrac{1}{Q_i}\cdot\dfrac{p_{i+1}}{q_{i+1}}=\dfrac{1}{Q_i}\cdot\dfrac{p'_{i+1}}{tq'_{i+1}}.$$
Therefore $\n^*(y_i)/\n^*(x_i)=p'_{i+1}/q'_{i+1}$.
\end{proof}

\begin{remark} Corollary \ref{claim} shows that for every $i\geq 0$ the
inclusion $R_{k_i} \subset S_{k'_i}$ has the stable form
(\ref{monomialequation}). However, the permissible parameters
 $(x_i,y_i)$ of $S_{k'_i}$ are not admissible if $q'_{i+1}=1$.

Notice also that by Lemma \ref{admissibleparameters} and Corollary
\ref{claim} for all $l\geq 1$ we have a commutative diagram
\begin{equation}
\begin{array}{lllllllllllll}
S_{s_{l-1}}&=&S_{k'_{i_{l-1}}}&\ra&S_{k'_{i_{l-1}+1}}&\ra&\dots&\ra&S_{k'_{i_l-1}}&\ra&S_{k'_{i_l}}&=&S_{s_{l}}\\
&&\uparrow&&\uparrow&&&&\uparrow&&\uparrow\\
R_{r_{l-1}}&=&R_{k_{i_{l-1}}}&\ra&R_{k_{i_{l-1}+1}}&\ra&\dots&\ra&R_{k_{i_l-1}}&\ra&R_{k_{i_l}}&=&R_{r_{l}}\\
\end{array}
\end{equation}
where the sequences satisfy the conclusions of Corollary
\ref{claim}.
\end{remark}


\end{document}